\documentclass[cmp]{svjour}
\usepackage{amsmath, amsfonts, amssymb}
\usepackage[pdftex]{graphicx}
\usepackage[pdftex,dvipsnames]{xcolor}
\usepackage{epsf}
\usepackage{psfrag}
\usepackage{enumerate}
\usepackage[hang,small,bf]{caption}
\usepackage{appendix}

\DeclareGraphicsExtensions{.ps,.eps}

\usepackage[latin1]{inputenc}

\numberwithin{equation}{section}

\newtheorem{bigthm}{Theorem}   

\newcommand{\sumtwo}[2]{\sum_{\substack{#1 \\ #2}}} 
\newcommand{\abs}[1]{\left| #1\right|}

\newcommand{\calA}{\mathcal{A}}
\newcommand{\calB}{\mathcal{B}}
\newcommand{\calC}{\mathcal{C}}
\newcommand{\calD}{\mathcal{D}}
\newcommand{\calE}{\mathcal{E}}
\newcommand{\calF}{\mathcal{F}}

\newcommand{\calM}{\mathcal{M}}
\newcommand{\calN}{\mathcal{N}}

\newcommand{\calP}{\mathcal{P}}
\newcommand{\calQ}{\mathcal{Q}}
\newcommand{\calR}{\mathcal{R}}

\newcommand{\calU}{\mathcal{U}}

\newcommand{\bbE}{\mathbb{E}}

\newcommand{\bbH}{\mathbb{H}}

\newcommand{\bbL}{\mathbb{L}}

\newcommand{\bbN}{\mathbb{N}}

\newcommand{\bbP}{\mathbb{P}}

\newcommand{\bbR}{\mathbb{R}}

\newcommand{\bbU}{\mathbb{U}}
\newcommand{\bbV}{\mathbb{V}}

\newcommand{\bbX}{\mathbb{X}}

\newcommand{\bbZ}{\mathbb{Z}}

\newcommand{\sfx}{{\sf x}}
\newcommand{\sfy}{{\sf y}}
\newcommand{\sfz}{{\sf z}}
\newcommand{\sfu}{{\sf u}}
\newcommand{\sfv}{{\sf v}}

\newcommand{\sfs}{{\sf s}}
\newcommand{\sfe}{{\sf e}}
\newcommand{\sfp}{{\sf p}}
\newcommand{\sfr}{{\sf r}}

\newcommand{\sfC}{{\sf C}}
\newcommand{\sfG}{{\sf G}}
\newcommand{\sfI}{{\sf I}}
\newcommand{\sfL}{{\sf L}}

\newcommand{\sfR}{{\sf R}}
\newcommand{\sfS}{{\sf S}}
\newcommand{\sfT}{{\sf T}}

\newcommand{\setof}[2]{\left\{#1 \,:\, #2 \right\}}
\newcommand{\given}{\,|\,}

\newcommand{\lb}{\left(}
\newcommand{\rb}{\right)}
\newcommand{\lbr}{\left\{}
\newcommand{\rbr}{\right\}}

\newcommand{\dd}{{\rm d}}

\newcommand{\step}[1]{S{\small TEP}\,#1.}

\newcommand{\Var}{\bbV{\rm ar}}

\newcommand{\1}{\mathbf{1}}
\newcommand{\smo}[1]{{\mathrm o}(#1)}

\newcommand{\be}[1]{\begin{equation}\label{#1}}
\newcommand{\ee}{\end{equation}}

\newcommand{\Hla}{H_\lambda}
\newcommand{\RWP}{\sfp}
\newcommand{\bsetof}[2]{\bigl\{#1 \,:\, #2 \bigr\}}
\newcommand{\sfw}{{\sf w}}
\DeclareMathOperator*{\slim}{{\sf s}-lim}

\newcommand{\eig}{\zeta}	

\begin{document}

\title{An invariance principle to Ferrari-Spohn diffusions}
\author{
Dmitry Ioffe\inst{1}
\thanks{DI was supported by the Israeli Science Foundation grants 817/09 
and 1723/14.}
\and
Senya Shlosman\inst{2}
\and
Yvan Velenik\inst{3}
\thanks{YV was partially supported by the Swiss National Science Foundation.}
}

\institute{
William Davidson Faculty of Industrial Engineering and Management,
Technion - Israel Institute of Technology,
Technion City, Haifa 32000,
Israel\\
\email{ieioffe@technion.ac.il}
\and
Aix Marseille Universit\'e, Universit\'e de Toulon,
CNRS, CPT UMR 7332, 13288, Marseille, France, and
Inst. of the Information Transmission Problems,
RAS, Moscow, Russia\\
\email{senya.shlosman@univ-amu.fr}
\and
Section de Math\'ematiques,
Universit\'e de Gen\`eve,
2-4, rue du Li\`evre,
Case postale 64,
1211 Gen\`eve\\
\email{Yvan.Velenik@unige.ch}
}
\date{}

\maketitle

\begin{abstract}
We prove an invariance principle for a class of  tilted $1+1$-dimensional SOS
models or, equivalently, for a class of tilted random walk bridges in $\bbZ_+$.
The limiting objects are stationary reversible ergodic diffusions with drifts
given by the logarithmic derivatives of the ground states of associated
singular Sturm-Liouville operators.  In the case of a linear area tilt,
we recover the Ferrari-Spohn diffusion with $\log$-Airy drift, which was derived
in~\cite{FS05} in the context of Brownian motions conditioned to stay above
circular and parabolic  barriers.
\end{abstract}


\section{Introduction and Results}

\subsection{Physical motivations}

We start with an informal description, in a restricted setting, of the
effective interface model at the core of our study; a detailed description in
the more general framework considered in the present work will be given in
Subsection~\ref{ssec:RW_with_Tilted_Areas}.

We consider a Gibbs random field $(X_i)_{1\leq i\leq N}$, with $X_i\in\bbN$ for
all $i$, and effective Hamiltonian
\[
\Hla = \sum_{i=1}^{N-1} \Phi (X_{i+1}-X_i) + \sum_{i=1}^N V_\lambda(X_i)\,,
\]
depending on a parameter $\lambda>0$. Later, we shall allow rather
general forms for the interaction $\Phi$ and for the external potential
$V_\lambda$. For the purpose of this introductory section, let us however
restrict the discussion to the physically very relevant case of
$V_\lambda(x) = \lambda x$, and assume that $\Phi$ is symmetric and grows fast
enough: for example, $\Phi(x)=x^2$, $\Phi(x)=|x|$ or $\Phi(x) = \infty\cdot
\1_{\abs{x}>R}$.
Let us denote by $\mu_{N;\lambda}$ the corresponding Gibbs measure with
boundary condition $X_1=X_N=0$.

With this choice, this model can be interpreted as follows. The random variable
$X_i$ models
the height of an interface above the site $i$. This interface separates an
equilibrium phase (above the interface) and a layer of unstable phase
(delimited by the interface and the wall located at height $0$). The parameter
$\lambda$ corresponds to the excess free energy associated to the unstable
phase.

Of course, when $\lambda=0$, the distribution of $X$ is just that of a
random walk, conditioned to stay positive,
with distribution of jumps given by
\be{eq:Phi-jumps}
p_x = \frac{{\rm e}^{-\Phi (x )}}{\sum_y {\rm e}^{-\Phi (y )}} .
\ee
In particular, the field delocalizes as
$N\to\infty$. When $\lambda>0$, however, the field remains localized uniformly
in $N$. We shall be mostly interested in the behavior as $\lambda\downarrow 0$
(say, either after letting $N\to\infty$, or by assuming that $N=N(\lambda)$
grows fast enough). In that case, one can prove that the typical width of the
layer is of order $\lambda^{-1/3}$ and that the correlation along the interface
is of order $\lambda^{-2/3}$~\cite{AS85,HV04}.

In the present work, we are interested in the scaling limit of the random field
$X$ as $\lambda\downarrow 0$, that is, in the limiting behavior of
$x_\lambda(t) = \lambda^{1/3}X_{[\lambda^{-2/3}t]}$. As stated in
Theorem~\ref{thm:A} below, in the particular case considered here, the scaling
limit is given by the diffusion on $(0,\infty)$ with generator
\[
\frac{\sigma^2}{2}\frac{\dd^2}{\dd r^2} + \sigma^2
\frac{\varphi_0^\prime }{\varphi_0}\frac{\dd}{\dd r}\,,
\]
where (see \eqref{eq:Phi-jumps}) $\sigma^2=\sum_x x^2 p_x$ and $\varphi_0 =
\mathsf{Ai}(\sqrt[3]{\frac{2}{\sigma^2}}\, x - \omega_1)$ with $-\omega_1$ the
first zero of the Airy function {\sf Ai}. This diffusion was first introduced
by Ferrari and Spohn in the context of Brownian motions conditioned to stay
above circular and parabolic  barriers~\cite{FS05}.

This scaling limit should be common to a wide class of systems, of which the
following are but a few examples:
\begin{itemize}
\item Critical prewetting in the 2d Ising model: behavior of the film of
unstable negatively magnetized layer induced by $(-)$-boundary
conditions, in the presence of a positive bulk magnetic field~\cite{V04}; see
Figure~\ref{fig:Ising}.
\item Interfacial adsorption at the interface between two equilibrium
phases~\cite{F84,SHK84}.
\item Geometry of the top-most layer of the $2+1$-dimensional SOS model above a
wall~\cite{CLMST12}.
\item Island of activity in kinetically constrained models~\cite{BLT12}.
\end{itemize}
\begin{figure}
\begin{center}
\includegraphics[width=5cm]{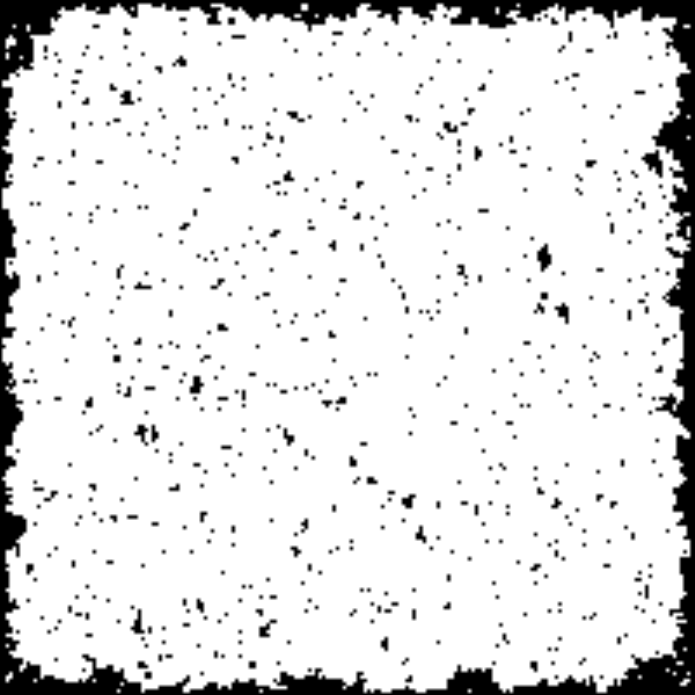}
\end{center}
\caption{A low-temperature two-dimensional Ising model in a box of sidelength
$N=200$ with negative boundary condition and a positive magnetic field of the
form $h=c/N$. When $c$ is above a critical threshold, the bulk of the system is
occupied by a positively magnetized phase, while the walls are wet by a film of
(unstable) negatively magnetized phase~\cite{SS96}. For a slightly different
geometry, it was shown in~\cite{V04} that this film has a width (along the
walls) of order $h^{-1/3+o(1)}$ as $h\downarrow 0$.}
\label{fig:Ising}
\end{figure}
\subsection{Limiting objects} Limiting objects are quantified in terms of 
Sturm-Liouville problems. 
\paragraph{A Sturm-Liouville problem}
The basic space we shall work with is $\bbL_2(\bbR_+)$.
The notations $\|\cdot\|_2$ and $\langle \cdot ,\cdot\rangle_2$ are reserved
for the corresponding norm and scalar product.
Given  $\sigma >0$ and   a non-negative function
$q\in \sfC^2 (\bbR_+ )$ which satisfies $\lim_{r\to\infty} q (r) = \infty$,
 consider the following family of singular  Sturm-Liouville operators on
$\bbR_+$:
\be{eq:SL-operators}
\sfL=\sfL_{\sigma,q} = \frac{\sigma^2}{2}\frac{\dd^2}{\dd r^2} - q (r)
,
\ee
with zero boundary condition $\varphi (0) = 0$.

\medskip
It is a classical result~\cite{CL55} that
$\sfL$ possesses a complete  orthonormal family $\lbr \varphi_i\rbr$  of
simple eigenfunctions
in $\bbL_2 \lb \bbR_+ \rb$ with eigenvalues
\be{eq:eigenv-ST}
0 >- \eig_0 > -\eig_1 > -\eig_2 > \dots ;\ \lim \eig_j = \infty.
\ee
The eigenfunctions $\varphi_i$ are smooth and $\varphi_i$ has exactly $i$
zeroes in $(0,\infty)$, $i=0,1,\ldots $.

The domain of (the closure of) $\sfL$ in $\bbL_2(\bbR_+)$ is
\be{eq:DomL}
\calD (\sfL ) = \lbr f = \sum_i a_i\varphi_i~ :~ \sum_i \eig_i^2 a_i^2 <\infty
\rbr
\quad
{\rm and}\quad  \sfL f = - \sum_i \eig_i a_i\varphi_i\ \text{for $f\in\calD$}.
\ee
Clearly, $\calD (\sfL )$ is dense in $\bbL_2(\bbR_+)$.
Indeed, since any function $f \in \bbL_2$ can be written as
$f = \sum_i a_i\varphi_i$, the linear space of all finite linear combinations
$\calU = \lbr \sum_{i=0}^{N} a_i \varphi_i \rbr \subset \calD(\sfL)$
is dense in $\bbL_2$. For any function $f= \sum_i a_i\varphi_i \in \calD (\sfL)$,
$\lim_{N\to\infty} \sfL\bigl(\sum_{i=0}^N a_i  \varphi_i\bigr) = \sfL f$. In
particular,
$\calU$ is a core for $\sfL$.

If $f = \sum_i a_i\varphi_i \in\calD$ and $\lambda > -\eig_0$, one has
\be{eq:DissipL}
\left\| \lb \lambda \sfI  -\sfL\rb f \right\|_2 = \left\|\sum_i  (\lambda +
\eig_i )a_i \varphi_i\right\|_2
\geq (\lambda +\eig_0 )\|f\|_2 ,
\ee
which shows that $\sfL  + \eig_0 \sfI$ is dissipative.

Furthermore, for any $\lambda > -\eig_0$,
${\rm Range}\lb \lambda \sfI -\sfL\rb = \bbL_2 (\bbR_+ )$,
and $\sfL$ has a compact resolvent
$\sfR_\lambda = \lb \lambda \sfI - \sfL\rb^{-1}$.
Indeed,
\[
\lb \lambda \sfI  -\sfL\rb
\sum_i\frac{a_i}{\lambda +\eig_i }\varphi_i  = \sum_i a_i\varphi_i
\quad{\rm and}\quad
 \sfR_\lambda\lb \sum_i a_i\varphi_i \rb = \sum_i \frac{a_i}{\lambda +\eig_i}
\varphi_i .
\]
By the Hille-Yosida theorem, $\sfL +\eig_0\sfI $ generates a strongly
continuous
contraction semigroup $\sfT^t$ on $\bbL_2 (\bbR_+ )$. Explicitly,
\be{eq:Tt}
\sfT^t \lb \sum_i a_i \varphi_i \rb = \sum_i \mathrm{e}^{- \lb \eig_i -\eig_0\rb
t}
a_i \varphi_i .
\ee

\paragraph{Ferrari-Spohn diffusions}
Define
\be{eq:FS-Gen}
\sfG_{\sigma, q} \psi =
 \frac{1}{\varphi_0 } \lb \sfL +\eig_0\rb\lb \psi\varphi_0 \rb =
\frac{\sigma^2}{2}\frac{\dd^2\psi }{\dd r^2} + \sigma^2
\frac{\varphi_0^\prime }{\varphi_0}\frac{\dd\psi }{\dd r} =
\frac{\sigma^2}{2\varphi_0^2}
\frac{\dd }{\dd r}\lb\varphi_0^2\frac{\dd \psi }{\dd r} \rb .
\ee
The sub-indices $\sigma$ and $q$ will be dropped whenever there is no risk of
confusion.
We shall say that $\sfG_{\sigma, q}$  is the  generator of a  Ferrari-Spohn
diffusion on $(0,\infty )$.
The diffusion itself is ergodic and reversible with respect to the measure
$\dd\mu_0(r)=\varphi_0^2(r)\dd r$. In the sequel, we shall denote by
$\sfS_{\sigma, q}^t$ the corresponding semigroup,
\be{eq:T-FS}
\sfS_{\sigma, q}^t\psi = \frac{1}{\varphi_0}\sfT^t(\psi\varphi_0),
\ee
and by $\bbP_{\sigma , q}$ the corresponding path measure.

\subsection{Random walks with tilted areas}
\label{ssec:RW_with_Tilted_Areas}
Let $p_y$ be an irreducible random walk kernel on $\bbZ$.
The probability of a finite trajectory $\bbX = (X_1, X_2, \ldots, X_k)$ is
$\sfp(\bbX) = \prod_i p_{X_{i+1} - X_i}$.
Let $\sfu,\sfv\in\bbN$ and $M,N\in\bbZ$ with $M\leq N$.
Let $\calP^{\sfu,\sfv}_{M,N,+}$ be the family of trajectories starting at
$\sfu$ at time $M$, ending at $\sfv$ at time $N$ and staying positive during
the time interval $\{M, \ldots ,N\}$.
For $N>0$, we shall use a shorthand notations
$\calP^{\sfu ,\sfv}_{N,+} = \calP^{\sfu ,\sfv}_{-N, N,+}$ and
$\hat\calP^{\sfu ,\sfv}_{N,+}= \calP^{\sfu ,\sfv}_{1,N,+}$.

\paragraph{Assumptions on $\sfp$}
Assume that
\begin{equation}
\label{eq:Assumption}
\sum_\sfz \sfz p_\sfz = 0 = {\sum_\sfz \sfz p_{-\sfz}}
\quad\text{and $\sfp$ has finite exponential moments}.
\end{equation}
In the sequel, we shall use the notation
\be{eq:sigma}
\sigma^2 = \sum_\sfz \sfz^2 p_\sfz = {\sum_\sfz \sfz^2 p_{-\sfz}} = \Var_{\sfp}
(X) <\infty .
\ee

\paragraph{The model}
Let $\lbr V_\lambda\rbr_{\lambda >0}$ be a family of self-potentials,
$V_\lambda :\bbR_+\to\bbR_+$.
Given $\lambda >0$, define the partition function
\be{eq:pf}
Z^{\sfu ,\sfv}_{N,+,\lambda} = \sum_{\bbX\in\calP^{\sfu ,\sfv}_{N,+}} {\rm
e}^{-\sum_{-N}^N
V_\lambda(X_i) }\sfp (\bbX) ,
\ee
and, accordingly, the probability distribution
\be{eq:pd}
\bbP^{\sfu,\sfv}_{N,+,\lambda} (\bbX) = \frac{1}{Z^{\sfu ,\sfv}_{N,+,\lambda}}
\mathrm{e}^{-\sum_{-N}^N V_\lambda (X_i) }\sfp (\bbX) .
\ee
The term $\sum_{-N}^N V_\lambda ( X_i )$ represents a generalized (non-linear)
area below
the trajectory $\bbX$. It reduces to (a multiple of) the usual area when
$V_\lambda (x) = \lambda x$. We make the following set of assumptions on
$V_\lambda $:

\paragraph{Assumptions on $V_\lambda$}
For any $\lambda>0$, the function $V_\lambda$ is
continuous monotone increasing
and satisfies
\be{eq:VL-1}
V_\lambda (0) = 0\quad{\rm and}\quad \lim_{x\to\infty} V_\lambda (x) = \infty .
\ee
In particular, the relation
\be{eq:HL-1}
\Hla^2 V_\lambda (\Hla ) = 1
\ee
determines unambiguously the quantity $\Hla$. Furthermore, we make the
assumptions that $\lim_{\lambda\to 0} \Hla =\infty$ and that
there exists a function $q\in \sfC^2 (\bbR^+ )$ such that
\be{eq:HL-2}
\lim_{\lambda\to 0} \Hla^2 V_\lambda (r \Hla ) = q (r),
\ee
uniformly on compact subsets of $\bbR_+$.
Note  that $\Hla$, respectively $\Hla^2$,  is the spatial, respectively
temporal, scale for the invariance principle which is formulated below in
Theorem~\ref{thm:A}.

Finally, we shall assume that there
exist
$\lambda_0 >0$ and a (continuous non-decreasing) function $q_0\geq 0$ with
$\lim_{r\to\infty} q_0 (r) = \infty$ such that, for all $\lambda\leq\lambda_0$,
\be{eq:HL-3}
\Hla^2 V_\lambda (r \Hla ) \geq q_0(r) \text{ on } \bbR_+ .
\ee

\subsection{Main result}
\label{sub:Main}
It will be convenient to think about $\bbX$ as being frozen outside $\{-N,
\ldots,N\}$. In this way, we can consider $\bbP^{\sfu ,\sfv}_{N,+ ,\lambda}$
as a distribution on the set of doubly infinite paths $\bbN^\bbZ$.

We set  $h_\lambda = \Hla^{-1}$. The paths are rescaled as follows:
For $t\in h_{\lambda}^{2}\bbZ$, define
\be{eq:xN}
x_\lambda (t) = h_{\lambda} X_{\Hla^2 t} = \frac{1}{H_\lambda}X_{\Hla^2 t}  .
\ee
$x_\lambda (t)$ is then extended to $t\in\bbR$ by linear interpolation. In
this way, we can talk about the induced  distribution of
$\bbP^{\sfu,\sfv}_{N,+,\lambda}$ on the space of continuous function $\sfC
[-T,T]$,
for any $T\geq 0$.
\begin{bigthm}
\label{thm:A}
Let $\lambda_N$ be a sequence satisfying $\lim_{N\to\infty}\lambda_N = 0$.
Assume, furthermore, that $\lim_{N\to\infty}
Nh_{\lambda_N}^2=+\infty$. Set $x_N(\cdot) =  x_{\lambda_N}(\cdot)$.
Fix any $c\in (0,\infty)$.
Then, as $N\to\infty$, the distributions of $x_N(\cdot)$ under
$\bbP^{\sfu,\sfv}_{N,+,\lambda}$ are, uniformly in $\sfu,\sfv\leq c
H_{\lambda_N}$, weakly convergent to the stationary  Ferrari-Spohn diffusion
$x_{\sigma,q}(\cdot)$ on $\bbR_+$ with the generator $\sfG_{\sigma,q}$
specified in~\eqref{eq:FS-Gen}.
\end{bigthm}
\begin{remark}
The condition $\lim_{N\to\infty}
Nh_{\lambda_N}^2=+\infty$ or, equivalently, $H_{\lambda_N}^2 \ll N$ has a
transparent meaning: $N$ is the size of the system, whereas
$H_{\lambda_N}^2$ is  the correlation length of the random
walk with $V_{\lambda_N}$-area tilts.
A precise statement of this sort is formulated in Proposition~\ref{prop:gs-lN}
in Subsection~\ref{sub:AsympGS} below.
\end{remark}
In the case $V_\lambda(x) = \lambda x$, the typical height $\Hla =
\lambda^{-\frac13}$, $q(r) = r$ and the ground state $\varphi_0$ is the
rescaled Airy function:
\be{eq:AiryGen}
\varphi_0 = {\sf Ai}(\chi r-\omega_1 ) \quad{\rm and}\quad e_0 =
\frac{\omega_1}{\chi},
\ee
where $-\omega_1 = -2.33811\ldots$ is the first zero of ${\sf Ai}$ and
$\chi = \sqrt[3]{\frac{2}{\sigma^2}}$. Indeed, for $\varphi_0$ defined as
in~\eqref{eq:AiryGen},
\[
\frac{\dd^2}{\dd r^2}{\sf Ai} (r) = r{\sf Ai}(r )\ \Rightarrow\
\frac{\dd^2}{\dd r^2}\varphi_0(r) = \chi^2 (\chi r-\omega_1)\varphi_0(r) ,
\]
and one only needs to tune $\chi$ in order to adjust to
the expression~\eqref{eq:SL-operators} for $\sfL$.

\section{Proofs}
The proof is a combination of probabilistic estimates based on an
early paper~\cite{HV04} and rather straightforward functional analytic
considerations. We shall first express the partition functions
$Z^{\sfu,\sfv}_{N,+,\lambda}$ and, subsequently, the probability distributions
$\bbP^{\sfu,\sfv}_{N,+,\lambda}$ in terms of powers of a transfer operator
$\sfT_\lambda$. For each $\lambda$, the operator $\sfT_\lambda$ gives rise,
via Doob's transform, to a stationary positive-recurrent Markov chain
$X^\lambda$ with path measure $\bbP_\lambda$. In the sequel, we shall refer to
$X^\lambda$ as to the \emph{ground-state chain}.
Following~\eqref{eq:xN}, the ground-state chains are rescaled as $x_\lambda (t)
= h_\lambda X^\lambda_{\Hla^2 t}$.

\medskip
The proof of Theorem~\ref{thm:A} comprises three steps:
\newline
\step{1} As $\lambda\to 0$, the finite-dimensional distributions of the
rescaled ground-state chains $x_\lambda$ converge to the finite-dimensional
distributions of the Ferrari-Spohn diffusion $x_{\sigma,q}$.
This is the content of Corollary~\ref{thm:ConvFDD} in Subsection~\ref{sub:ConvEigen}.
\newline
\step{2} Under our assumptions on $\sfp$ and on the family of potentials
$V_\lambda$, the induced family of distributions $\bbP_\lambda$ is tight on
$\sfC [-T,T]$ for any $T <\infty$. This is the content
of Proposition~\ref{prop:tightness} in Subsection~\ref{sub:tightness}.
\newline
\step{3} Under the conditions of Theorem~\ref{thm:A}, the following happens:
For each $T\geq 0$, the variational distance between the induced 
distributions
on $\sfC[-T,T]$ of $\bbP^{\sfu,\sfv}_{N,+,\lambda_N}$ and of $\bbP_{\lambda_N}$
tends to zero as $N\to\infty$. This is the content
of Corollary~\ref{prop:gs-structure} in Subsection~\ref{sub:AsympGS}.

\paragraph{Remark on constants}
$c_1,\nu_1,\kappa_1,c_2,\nu_2,\kappa_2,\ldots $ denote positive constants
which may take different values in different Subsections, but are otherwise
universal, in the sense that the corresponding bounds hold uniformly in the
range of the relevant parameters.

\subsection{The operator $\sfT_\lambda$ and its Doob transform}
\label{sub:Tlambda}
In the sequel, we shall make a slight abuse of notation and identify the
spaces $\ell_p(\bbN)$ with sub-spaces of $\ell_p(\bbZ)$:
\[
\ell_p(\bbN)
=
\setof{\phi\in\ell_p(\bbZ)}{\phi(\sfx)=0 \text{ for } \sfx\leq 0} .
\]
For $\lambda >0$, consider the operators $\tilde \sfT_\lambda$ defined by
\be{eq:tTlambda}
\tilde\sfT_\lambda [\phi] (\sfx) =
\sum_{\sfy} p_{\sfy-\sfx} \mathrm{e}^{-\frac{1}{2}
(V_\lambda(\sfx)+V_\lambda(\sfy))}
\phi(\sfy) .
\ee
In terms of $\tilde \sfT_\lambda$, the partition function can be expressed as
\be{eq:ReprTlambda}
\mathrm{e}^{\frac{1}{2} (V_\lambda (\sfu)+ V_\lambda (\sfv))}
Z^{\sfu,\sfv}_{N,+,\lambda} = \tilde \sfT_\lambda^{2N}[\1_\sfv](\sfu) .
\ee
For each $\lambda >0$, the operator $\tilde\sfT_\lambda$ is positive on
$\ell_p(\bbN)$ and compact from $\ell_p(\bbN)$ to $\ell_{q}(\bbN)$ for every
$p,q\in [1,\infty]$ (we use $\ell_\infty(\bbN)$ for the Banach space of
functions $\phi\in\bbR^\bbN$ which tend to zero as $\sfx\to\infty$). Indeed,
$\{
\1_{\sfx}\}_{\sfx\in\bbN}$ is a basis of $\ell_p(\bbN)$. Since,
for $\sfx , \sfy >0$,
\[
\tilde\sfT_\lambda [ \1_\sfx ] (\sfy ) =
\mathrm{e}^{-\frac{1}{2} ( V_\lambda(\sfx)+V_\lambda(\sfy) )}
p_{\sfx-\sfy} ,
\]
it follows that
\be{eq:T-form}
\tilde\sfT_\lambda[\1_\sfx]
=
\mathrm{e}^{-\frac{1}{2} V_\lambda(\sfx)}
\sum_\sfy  p_{\sfx-\sfy}\mathrm{e}^{-\frac{1}{2}V_\lambda(\sfy)} \1_\sfy \
\Rightarrow \ \|\tilde\sfT_\lambda[\1_\sfx]\|_p \leq \mathrm{e}^{-\frac{1}{2} V_\lambda(\sfx)}
\sum_\sfy  p_{\sfx-\sfy}\mathrm{e}^{-\frac{1}{2}V_\lambda(\sfy)} .
\ee
Hence, by the second condition in~\eqref{eq:VL-1} and by the
assumption on exponential tails of $\sfp$ in~\eqref{eq:Assumption},
the closure $\{\tilde\sfT_\lambda[\1_\sfx]\}$ is compact in any $\ell_{q} (\bbN)$
whenever $\lambda >0$.

Since $\tilde\sfT_\lambda$ is a positive operator (and since, e.g.,
$\sum_n 2^{-n}\tilde \sfT_\lambda^n$ is strictly positive and still compact),
the Krein-Rutman Theorem \cite[Theorem~6.3]{KR50} applies, and 
$\tilde\sfT_\lambda$ possesses a strictly
positive leading eigenfunction $\phi_\lambda$ (the same for all $\ell_p (\bbN)$
spaces, by compact embedding) of algebraic multiplicity one.
Let $E_\lambda$ be the corresponding leading eigenvalue.
All the above reasoning applies to the adjoint operator $\tilde\sfT_\lambda^*$
with matrix entries
\be{eq:T-l-adjoint}
\tilde\sfT_\lambda^* (\sfx,\sfy) = \tilde\sfT_\lambda (\sfy,\sfx) =
\mathrm{e}^{-\frac{1}{2}(V_\lambda(\sfx) + V_\lambda(\sfy))}  p_{\sfx-\sfy}.
\ee
Let $\phi^*_\lambda$ be the Krein-Rutman eigenfunction (with the very same
leading eigenvalue $E_\lambda$) of $\tilde\sfT_\lambda^*$.
As Theorem~\ref{thm:A} indicates, the relevant spatial scale is given
by $h_\lambda = \Hla^{-1}$.
To fix notation, we normalize $\phi_\lambda$ and $\phi_\lambda^*$ as
in~\eqref{eq:Normalization} below,
that is $h_\lambda\sum_x \phi_\lambda (x )^2 = h_\lambda\sum_x \lb 
\phi_\lambda^* (x )\rb^2 =1$.

It will be convenient to work with the following normalized version
$\sfT_\lambda$ of $\tilde\sfT_\lambda$:
for $\sfx,\sfy\in\bbN$, let us introduce the kernel
\be{eq:Tlambda}
\sfT_\lambda (\sfx,\sfy) = \frac{1}{E_\lambda}\tilde \sfT_\lambda(\sfx,\sfy)
= \frac{1}{E_\lambda}
\mathrm{e}^{-\frac{1}{2}(V_\lambda(\sfx) + V_\lambda(\sfy))} p_{\sfy-\sfx } .
\ee
In this way, $\phi_\lambda$ and $\phi^*_\lambda$ are the principal positive
Krein-Rutman eigenfunctions of $\sfT_\lambda$ and $\sfT_\lambda^*$
with eigenvalue $1$.

\paragraph{Ground-state chains}
Define
\be{eq:pi-lambda}
\pi_\lambda (\sfx,\sfy) =
\frac{1}{\phi_\lambda(\sfx)} \sfT_\lambda(\sfx,\sfy) \phi_\lambda(\sfy)
\quad \text{and}\quad
\pi_\lambda^*(\sfx,\sfy) =
\frac{1}{\phi_\lambda^*(\sfx)} \sfT_\lambda^*(\sfx,\sfy) \phi_\lambda^*(\sfy) .
\ee
$\pi_\lambda$ and $\pi_\lambda^*$ are irreducible Markov kernels on $\bbN$.
The corresponding chains are positively recurrent and have the invariant
probability measure
$\mu_\lambda(\sfx) = {c_\lambda}\phi^*_\lambda(\sfx)\phi_\lambda(\sfx)$.
As we shall prove below in Theorem~\ref{thm:phihat},
$\lim_{\lambda\to 0}\frac{h_\lambda}{c_\lambda} = 1$.

Notice that
\be{eq:pi-reversed}
\mu_\lambda(\sfx) \pi_\lambda(\sfx,\sfy) =
\mu_\lambda(\sfy) \pi^*_\lambda(\sfy,\sfx) =
{c_\lambda} \phi^*_\lambda(\sfx) \sfT_\lambda(\sfx,\sfy) \phi_\lambda(\sfy) .
\ee
In the sequel, we shall denote by $\bbP_{\lambda}$ the stationary distribution
on $\bbN^\bbZ$ of the (direct) ground-state chain which corresponds to
$\pi_\lambda$.

\paragraph{Variational description of $E_\lambda$ and $\mu_\lambda$}
Let us formulate a Donsker-Varadhan type formula for the principal eigenvalue
$E_\lambda$ of $\tilde\sfT_\lambda$ or, equivalently, for the eigenvalue $1$ of
$\sfT_\lambda$.
\begin{theorem}
\label{thm:DV-lambda}
\be{eq:DV-mu-lambda}
1 = \sup_{\mu}\inf_{u\in\bbU_{+}} \sum_{\sfx\in\bbN}
\mu(\sfx)\frac{\sfT_\lambda[u]}{u}(\sfx) .
\ee
\end{theorem}
Above, the first supremum is over probability measures on $\bbN$, and
\be{eq:U-lambda}
\bbU_{+} = \setof{u = v\phi_\lambda}{0 <\inf v \leq \sup v<\infty} .
\ee
\begin{remark}
Eventually, our proof of Theorem~\ref{thm:A} will not rely on the variational
formula~\eqref{eq:DV-mu-lambda}. Theorem~\ref{thm:DV-lambda} and its
consequence, Proposition~\ref{prop:F-lambda}, are formulated and proved in
their own right, but also because they elucidate the type of variational
convergence behind Theorem~\ref{thm:phihat} below.
\end{remark}

\begin{proof}
As before, set $h_\lambda = \Hla^{-1}$ and consider the following
functional:
\be{eq:F-lambda}
\calF_\lambda (\mu) =
\frac{1}{h_\lambda^{2}} \sup_{u\in\bbU_+} \sum_{\sfx\in\bbN} \mu(\sfx)
\frac{(1-\sfT_\lambda)u}{u}(\sfx) .
\ee
The coefficient $h_\lambda^{2}$ plays no role in the proof, it just fixes the
proper scaling.
The claim of Theorem~\ref{thm:DV-lambda} will follow once we show that
\be{eq:ToProve-lambda}
\calF_\lambda (\mu_\lambda) = 0 \quad\text{and}\quad \calF_\lambda (\mu)>0
\quad\text{whenever }\mu\neq\mu_\lambda .
\ee
Taking $u\equiv\phi_\lambda$, we readily infer that $\calF_\lambda$ is
non-negative.
In order to check the first statement in~\eqref{eq:ToProve-lambda}, we need to
verify that
\be{eq:claim1-lambda}
\sum_{\sfx\in\bbN} \mu_\lambda(\sfx) \frac{(1-\sfT_\lambda)u}{u}(\sfx) \leq 0 ,
\ee
whenever $u = v \phi_\lambda$ and $v\not\equiv 1$.
Let us write $v$ as $v = \mathrm{e}^{h}$. Then, using the
notation~\eqref{eq:pi-lambda},
\[
\sum_{\sfx\in\bbN} \mu_\lambda(\sfx) \frac{(1-\sfT_\lambda)u}{u}(\sfx)  =
\sum_{\sfx\in\bbN} \bigl(1 - \mathrm{e}^{-h}\pi_\lambda \mathrm{e}^h \bigr)
\mu_\lambda(\sfx)
\leq
\bigl( 1 - \mathrm{e}^{-\langle \mu_\lambda , h\rangle +
\langle \mu_\lambda ,\pi_\lambda  h\rangle } \bigr) = 0,
\]
where we used again Jensen's inequality and the invariance of 
$\mu_\lambda$:
$\mu_\lambda\pi_\lambda = \mu_\lambda$.

If $\mu = g^2\mu_\lambda$ and $g$ is bounded away from $0$ and $\infty$, then,
taking $u = g\phi_\lambda$,
\be{eq:g-comp-lambda}
\sum_{\sfx\in\bbN} \mu(\sfx) \frac{(1-\sfT_\lambda)u}{u}(\sfx) =
\frac{1}{2}\sum_{\sfx,\sfy} \hat\pi_\lambda(\sfx,\sfy) (g(\sfx)-g(\sfy))^2
\mu_\lambda(\sfx) ,
\ee
where $\hat\pi_\lambda$ is the symmetrized kernel,
\be{eq:hat-pi-lambda}
\hat\pi_\lambda(\sfx,\sfy) = \frac{1}{2}(\pi_\lambda(\sfx,\sfy) +
\pi^*_\lambda(\sfx,\sfy)) .
\ee
By \eqref{eq:pi-reversed},
\be{eq:g-bound}
\calF_\lambda (g^2\mu_\lambda) \geq
\frac{{c_\lambda}}{4} \sum_{\sfx,\sfy}
\Bigl( \frac{g(\sfx)-g(\sfy)}{h_\lambda}\Bigr)^2
(\phi^*_\lambda(\sfx)\phi_\lambda(\sfy) +
\phi^*_\lambda(\sfy)\phi_\lambda(\sfx)) .
\ee
We claim that~\eqref{eq:g-bound} still holds when $g$ is not bounded away from
$0$ and $\infty$.
This will follow if we show that there exists a sequence $u_\ell\in\bbU_+$
such that
\begin{align}
\limsup_{\ell\to \infty}
\sum_{\sfx\in\bbN} \mu_\lambda(\sfx) g^2(\sfx)
\frac{\sfT_\lambda [u_\ell\phi_\lambda](\sfx)}{u_\ell(\sfx)\phi_\lambda(\sfx)}
&\leq
\sum_{\sfx\in\bbN} \mu_\lambda(\sfx) g^2(\sfx)
\frac{\sfT_\lambda [g\phi_\lambda](\sfx)}{g(\sfx)\phi_\lambda(\sfx)}\notag\\
&=
c_\lambda \sum_{\sfx\in\bbN} \phi^*_\lambda(\sfx) g(\sfx)
\sfT_\lambda[g \phi_\lambda](\sfx).
\label{eq:u-g-lim}
\end{align}
Assume that $g$ is not bounded away from zero and consider $g_n = g\vee
\frac{1}{n}$. Then,
\[
\sum_{\sfx\in\bbN} \mu_\lambda(\sfx) g^2(\sfx)
\frac{\sfT_\lambda[g_n\phi_\lambda](\sfx)}{g_n(\sfx)\phi_\lambda(\sfx)}
\leq
{c_\lambda} \sum_{\sfx\in\bbN} \phi^*(\sfx)g(\sfx)
\sfT_\lambda[g_n\phi_\lambda](\sfx) .
\]
By a monotone convergence argument, the right-hand side above converges (as
$n\to\infty$) to ${c_\lambda} \sum_{\sfx\in\bbN} \phi^*(\sfx) g(\sfx)
\sfT_\lambda[g\phi_\lambda](\sfx)$.
If, in addition, $g$ is not bounded above, then consider
$g_{n,M} = g_n \wedge M\in\bbU_+$. Define $A_M = \setof{\sfx}{g(\sfx)>M}$.
Then,
\begin{equation*}
\sum_{\sfx\in\bbN} \mu_\lambda(\sfx) g^2(\sfx)
\frac{\sfT_\lambda[g_{n,M}\phi_\lambda](\sfx)}{g_{n,M}(\sfx)\phi_\lambda(\sfx)}
\leq
\sum_{\sfx\not\in A_M} \mu_\lambda(\sfx) g^2(\sfx)
\frac{\sfT_\lambda[g_n\phi_\lambda](\sfx)}{g_n(\sfx)\phi_\lambda(\sfx)}
+ \sum_{\sfx\in A_M} \mu_\lambda(\sfx) g^2(\sfx) .
\end{equation*}
Since $\lim_{M\to\infty} \sum_{\sfx\in A_M} \mu_\lambda(\sfx) g^2(\sfx) = 0$,
the approximation procedure goes through as claimed in~\eqref{eq:u-g-lim}.
\qed
\end{proof}
As a byproduct, we obtain the following
\begin{proposition}
\label{prop:F-lambda}
The functional $\calF_\lambda$ is convex and lower-semicontinuous. It has a
unique minimum:
\be{eq:F-lambda-min}
\calF_\lambda (\mu ) = 0\  \Leftrightarrow\ \mu = \mu_\lambda .
\ee
Furthermore, $\mu_\lambda$ is a quadratic minimum in the sense
of~\eqref{eq:g-bound}.
\end{proposition}

\subsection{Convergence of eigenfunctions, invariant measures and semigroups}
\label{sub:ConvEigen}
It will be convenient to think about $\sfT_\lambda$ and $\pi_\lambda$ as
acting on the rescaled spaces $\ell_2(\bbN_\lambda)$, where
\be{eq:l2-L}
\bbN_\lambda = h_\lambda \bbN\quad \text{and the scalar product is}\quad
\langle u, v\rangle_{2,\lambda}
=
h_\lambda \sum_{\sfr\in\bbN_\lambda } u(\sfr) v(\sfr) .
\ee
Accordingly, we rescale $\phi_\lambda$ and $\phi^*_\lambda$ in such a way that
\be{eq:Normalization}
\|\phi_\lambda\|_{2,\lambda} = \|\phi_\lambda^*\|_{2,\lambda}=1 .
\ee
We use the same notation $\mu_\lambda = c_\lambda\phi_\lambda \phi_\lambda^*$
for the rescaled probability measure on $\bbN_\lambda$.
In other words, the constants $c_\lambda$ are defined via
\be{eq:c-lambda}
\frac{1}{c_\lambda}
=
\sum_{\sfr\in\bbN_\lambda}\phi_\lambda(\sfr)\phi_\lambda^*(\sfr)
\quad\text{or}\quad
\frac{h_\lambda}{c_\lambda}
=
\langle \phi_\lambda,\phi_\lambda^* \rangle_{2,\lambda} ,
\ee
where $\phi_\lambda$ and $\phi_\lambda^*$ are the principal eigenfunctions
satisfying the normalization condition~\eqref{eq:Normalization}.
\begin{remark}
\label{rem:NZ}
As in the case of $\ell_p(\bbN)$, with a slight abuse of notation,
we shall identify $\ell_2(\bbN_\lambda)$ with a closed linear sub-space of
$\ell_2(\bbZ_\lambda)$, where $\bbZ_\lambda = h_\lambda\bbZ$. Namely,
\be{eq:NZ}
\ell_2(\bbN_\lambda) = \setof{u\in\ell_2(\bbZ_\lambda)}{u(r)=0\text{ for all
$r\leq 0$}} .
\ee
In this way, if $k_\lambda$ is a kernel on $\bbZ_\lambda$, then the operators
\[
u(\cdot)\mapsto \1_{\cdot\in\bbN_\lambda}\sum_{\sfs\in\bbZ_\lambda} k_\lambda
(\sfs - \cdot)u(\sfs)\quad \text{and}\quad
u(\cdot)\mapsto \1_{\cdot\in\bbN_\lambda}
\sum_{\sfs\in\bbZ_\lambda} k_\lambda(\cdot -\sfs) (u(\sfs) - u (\cdot)) ,
\]
can be considered as operators on $\ell_2(\bbN_\lambda)$. Accordingly,
\[
\sum_{\sfs,\sfr\in\bbZ_\lambda} k_\lambda(\sfs-\sfr) (u(\sfs) - u(\sfr)) u(\sfr)
\]
is a quadratic form on $\ell_2(\bbN_\lambda)$.
\end{remark}

In the sequel, we shall write $p_\lambda(\sfr) = p_{\Hla\sfr}$ for the
rescaled random walk kernel on $\bbZ_\lambda$.

\paragraph{Convergence of Hilbert spaces}
Let us fix a map $\rho_\lambda:\bbL_2 (\bbR_+)\to \ell_2(\bbN_\lambda)$ with
$\|\rho_\lambda\|\leq 1$. The specific choice is not really important; for
instance, we may define
\be{eq:rho-lambda}
\rho_\lambda u(\sfr) = \tfrac1{h_\lambda}\int_{\sfr-h_\lambda}^\sfr u (s)\dd s .
\ee
\begin{definition}
\label{def:ConvHS}
Let us say that a sequence $u_\lambda \in \ell_2(\bbN_\lambda)$ converges to
$u\in\bbL_2(\bbR_+)$, $u = \slim u_\lambda$,
if
\be{eq:L-conv}
\lim_{\lambda\to 0} \| u_\lambda - \rho_\lambda u\|_{2,\lambda} = 0 .
\ee
\end{definition}
We shall write $\lim_{\lambda\to 0}$ instead of $\slim_{\lambda\to 0}$ whenever
no ambiguity arises.
\paragraph{Compactness of eigenfunctions}
The following two probabilistic estimates will be proved in
Section~\ref{sec:tools}.
\begin{lemma}
\label{lem:elambda}
Define $\sfe_\lambda = - \Hla^2\log E_\lambda$. Then,
\be{eq:elambda}
0 < \liminf_{\lambda\to 0} \sfe_\lambda
\leq \limsup_{\lambda\to 0} \sfe_\lambda < \infty .
\ee
\end{lemma}
As we already noted, it follows from the compactness of $\sfT_\lambda$ that the
eigenfunctions $\phi_\lambda$ and $\phi_\lambda^*$ belong to
$\ell_p(\bbN_\lambda)$ for any $p\geq 1$ and $\lambda>0$.
The second probabilistic input is a tail estimate on $\phi_\lambda$ and
$\phi_\lambda^*$.
\begin{lemma}
\label{lem:phi-l-K}
There exist positive constants $\nu_1$ and $\nu_2$  such that
\be{eq:tail-phi-l}
h_\lambda \sum_{\sfr\in\bbN_\lambda} 
\phi_\lambda(\sfr)\1_{\{\sfr>K\}}
\leq \nu_1e^{-\nu_2 KH_\lambda
(\sqrt{V_\lambda (H_\lambda K )}\wedge 1)}\leq
\nu_1 \mathrm{e}^{-\nu_2 K ( \sqrt{q_0(K)} \wedge \Hla ) } ,
\ee
uniformly in $K>0$ and $\lambda\leq\lambda_0$.
The same holds for $\phi_\lambda^*$.
In particular, both sequences $\phi_\lambda$ and $\phi_\lambda^*$ are bounded
in $\ell_1(\bbN_\lambda)$:
\be{eq:ell-1-bound}
\limsup_{\lambda\to 0}h_\lambda \sum_{\sfr\in\bbN_\lambda} \phi_\lambda(\sfr)
<\infty
\quad\text{and}\quad
\limsup_{\lambda\to 0}h_\lambda \sum_{\sfr\in\bbN_\lambda} \phi_\lambda^*(\sfr)
<\infty .
\ee
\end{lemma}
Using the two lemmas above and Rellich's theorem (see, e.g., 
\cite[Chapter~6]{Adams}) on compact embeddings of the Sobolev spaces 
$\bbH^1[a,b]$ into $\bbL_2[a,b]$ for finite intervals
$[a,b]$, we shall prove the following
\begin{proposition}
\label{prop:Comp}
Under our assumptions on $V_\lambda$ and $\sfp$, the sequence
$\phi_\lambda$ is sequentially compact (in the sense of {\sf s}-convergence
as described above in~\eqref{eq:L-conv}).
\end{proposition}
\begin{proof}
The proof comprises two steps: We first show that we can
restrict attention to the compactness properties of the functions $\psi_\lambda$
defined in~\eqref{eq:psi-l-def} below. We then check that the sequence
$\psi_\lambda$ satisfies the energy-type estimate~\eqref{eq:reduction2}, which
enables a uniform control of both tails of $\psi_\lambda$ and of their Sobolev
norms over $\bbR_+$. In this way, sequential compactness follows by a standard
diagonal argument.

\smallskip
\noindent\step{1}
In view of Lemma~\ref{lem:phi-l-K}, rather than studying directly the functions
$\phi_\lambda$, we can instead study the convergence properties of the functions
\be{eq:psi-l-def}
\psi_\lambda (\sfr) =
\mathrm{e}^{-{\frac{1}{2}} V_\lambda (\Hla\sfr)} \phi_\lambda(\sfr) .
\ee
Indeed, by~\eqref{eq:tail-phi-l}, there exists a sequence $\delta_\lambda\to 0$
such that
\be{eq:phi-l-Hl}
\lim_{\lambda\to 0}  h_\lambda \sum_{\sfr\in\bbN_\lambda}
\phi_\lambda^2(\sfr) \1_{\{V_\lambda(\Hla\sfr)>\delta_\lambda\}}
= 0 .
\ee
So, \eqref{eq:phi-l-Hl} implies that the norm of the difference
$\|\phi_\lambda - \psi_\lambda\|_{2,\lambda}$ tends to zero, and hence
$\phi_\lambda - \psi_\lambda$ tends to zero in the sense of
Definition~\ref{def:ConvHS}.

\smallskip
\noindent
\step{2}
In terms of $\psi_\lambda$, the eigenvalue equation
$\tilde\sfT_\lambda\phi_\lambda = E_\lambda\phi_\lambda$ reads as (recall
Remark~\ref{rem:NZ})
\be{eq:EigEq-psi}
\sum_{\sfs\in\bbZ_\lambda} p_\lambda(\sfs-\sfr)
(\psi_\lambda(\sfs) - \psi_\lambda(\sfr)) =
(E_\lambda \mathrm{e}^{V_\lambda (\Hla\sfr)} - 1) \psi_\lambda(\sfr) .
\ee
Multiplying both sides by $-\psi_{\lambda}(\sfr)$ and summing over $\sfr$, we
get
\[
h_{\lambda} \sum_{\sfr,\sfs\in\mathbb{Z}_{\lambda}} p_{\lambda}(\sfs-\sfr)
\Bigl(
\frac{-\psi_{\lambda}(\sfr)\psi_{\lambda}(\sfs) +\psi_{\lambda}^{2}(\sfr)}
{h_{\lambda}^{2}} \Bigr) +
h_{\lambda}
\sum_{\sfr\in\mathbb{N}_{\lambda}}
\frac{E_{\lambda}e^{V_{\lambda}(\Hla\sfr)}-1}{h_{\lambda}^{2}}
\psi_{\lambda}^{2}(\sfr)
= 0 .
\]
So, for the symmetrized kernel
$\hat p_\lambda(\sfz) = (p_\lambda(\sfz) + p_\lambda(-\sfz))/2$, we obtain
\be{eq:DF-V}
h_{\lambda} \sum_{\sfr,\sfs\in\bbZ_\lambda} \hat p_\lambda(\sfs-\sfr)
\Bigr(
\frac{\psi_\lambda(\sfs) - \psi_\lambda(\sfr)}{h_\lambda}\Bigr)^2
+ h_\lambda \sum_{\sfr\in\bbN_\lambda}
\frac{E_\lambda \mathrm{e}^{V_\lambda(\Hla\sfr)}-1}{h_\lambda^2}
\psi_\lambda^2(\sfr)
= 0 .
\ee
In view of Lemma~\ref{lem:elambda}, we may assume that there exists
$\bar\sfe<\infty$ such that, possibly going to a subsequence, the limit
$\sfe=\lim_{\lambda\to 0}\sfe_\lambda$ exists and satisfies
$\sfe<\bar\sfe$. So, we may assume that
$E_\lambda \geq \mathrm{e}^{-\bar\sfe h_\lambda^2}$. Recall also
our assumption~\eqref{eq:HL-3} on the growth of $V_\lambda$.
Let $\bar\sfr = \sup\setof{\sfr}{q_0(\sfr) < \bar\sfe}$.
Then, \eqref{eq:DF-V} implies that
\be{eq:reduction1}
h_\lambda \sum_{\sfr,\sfs\in\bbZ_\lambda}
\hat p_\lambda(\sfs-\sfr)
\Bigl( \frac{\psi_\lambda(\sfs) - \psi_\lambda(\sfr)}{h_\lambda}\Bigr)^2
+ h_\lambda \sum_{\sfr\geq\bar\sfr}
q_0(\sfr) \psi_\lambda^2(\sfr)
\leq
\bar\sfe \|\psi_\lambda\|_{2,\lambda}^2 .
\ee
By construction, $\|\psi_\lambda\|_{2,\lambda}^2 \leq 1$ and, as we have
already mentioned, \eqref{eq:phi-l-Hl} implies that actually
$\lim_{\lambda\to 0} \|\psi_\lambda\|_{2,\lambda}^2 = 1$.

Furthermore, since $\sfp$ is an irreducible kernel, there exists $\delta>0$
and a finite sequence of integer states $\sfx_0,\sfx_1,\ldots,\sfx_n$ with
$\hat p_{\sfx_i-\sfx_{i-1}}\geq\delta$, which connects $\sfx_0=0$ to
$\sfx_n = 1$.
Therefore,
\be{eq:sum-n2}
\sum_{\sfr,\sfs\in\bbZ_\lambda} \hat p_\lambda(s-r)
\Bigl( \frac{\psi_\lambda(\sfs) - \psi_\lambda(\sfr)}{h_\lambda}\Bigr)^2
\geq
\frac{\delta}{ n^2   } \sum_{\sfr\in\bbN_\lambda}
\Bigl(\frac{\psi_\lambda(\sfr)-\psi_\lambda(\sfr-h_\lambda)}{h_\lambda}\Bigr)^2,
\ee
where we use the elementary  inequality
$$
(z_0-z_1)^2 + (z_1-z_2)^2 + \ldots + (z_{n-1}-z_n)^2 \geq
\frac{1}{n} (z_0-z_n)^2,
$$
valid for all real $z_i$. 
The additional $1/n$ in the prefactor $1/n^2$ in \eqref{eq:sum-n2} is due to the 
fact that each term $\lb{\psi_\lambda(\sfs) - \psi_\lambda(\sfr)}\rb^2$ is used
in this way  at most $n$ times.
Together with~\eqref{eq:reduction1}, this implies the existence of two finite
positive constants $c_1$ and $c_2$ such that
\be{eq:reduction2}
c_1 h_\lambda
\sum_{\sfr\in\bbN_\lambda} \Bigl(
\frac{\psi_\lambda(\sfr) - \psi_\lambda(\sfr-h_\lambda )}{h_\lambda}\Bigr)^2
+ h_\lambda \sum_{\sfr\geq\bar\sfr} q_0(\sfr) \psi_\lambda^2(\sfr)
\leq c_2 .
\ee
This is the desired energy estimate, which holds for all $\lambda>0$ small.

\smallskip
\noindent
The rest of the proof is straightforward.
Let $\Psi_\lambda$ be the linear interpolation of $\psi_\lambda$:
for $\sfr\in\bbN_\lambda\cup\{0\}$ and $t\in [0,1]$,
\[
\Psi_\lambda (\sfr + th_\lambda) =
(1-t)\psi_\lambda (\sfr) + t \psi_\lambda (\sfr +h_\lambda) .
\]
The relation~\eqref{eq:reduction2} and $\lim_{\sfr\to\infty} q_0(\sfr)=\infty$
imply that $\lim_{n\to\infty} \| \Psi_\lambda \1_{\{r>n\}}\|_2 = 0$,
uniformly in $\lambda$ small. On the other hand, the very
same~\eqref{eq:reduction2} and Rellich's compact embedding theorem imply that,
for any $n<\infty$, the family $\Psi_\lambda \1_{\{r\leq n\}}$ is
subsequentially compact in $\bbL_2[0,n]$. Alternatively, \eqref{eq:reduction2}
implies that the linear interpolations $\Psi_\lambda$ are uniformly continuous
on $[0,n]$ for each $n$ fixed.
We conclude that the family $\Psi_\lambda$ is subsequentially compact in
$\bbL_2(\bbR_+)$.

Remember how the map $\rho_\lambda$ was defined in~\eqref{eq:rho-lambda}.
Since $\Psi_\lambda$ is the linear interpolation of $\psi_\lambda$,
and since $\lim_{\lambda\to 0} h_\lambda = 0$, the energy
estimate~\eqref{eq:reduction2} evidently implies that
\[
\lim_{\lambda\to 0}\| \psi_\lambda - \rho_\lambda \Psi_\lambda \|_{2,\lambda}
= 0 .
\]
Hence, $\psi_\lambda$ is subsequentially compact as well.
\qed
\end{proof}

\paragraph{Convergence of semigroups}
Possibly going to a subsequence, we can assume that $\lim\sfe_\lambda=\sfe$.
We shall rely on Kurtz's semigroup convergence
theorem~\cite[Theorem~I.6.5]{EK}: Define
\be{eq:L-lambda}
\sfL_\lambda f (\sfr) =
\frac{\sfT_\lambda - \sfI}{h_\lambda^2} f(\sfr) .
\ee
The following two statements are equivalent:
\begin{enumerate}[(a)]
\item For any $u\in\calU$, one can find a sequence
$u_\lambda\in\ell_2(\bbN_\lambda)$ such that both
$\lim_{\lambda\to 0} u_\lambda = u$ and
$\lim_{\lambda\to 0} \sfL_\lambda u_\lambda = (\sfL +\sfe)u$.

\item If $\lim_{\lambda\to 0} f_\lambda = f$, then
$\lim_{\lambda\to 0} \sfT_\lambda^{\lfloor \Hla^2 t\rfloor}f_\lambda =
\mathrm{e}^{(\sfL+\sfe\sfI)t}f$.
\end{enumerate}
The above equivalence holds provided that the operators $\sfT_\lambda$ are
linear contractions (which is straightforward), and that
$\mathrm{e}^{(\sfL+\sfe\sfI)t}$ is a strongly continuous semigroup with
generator $\sfL +\sfe\sfI$, but that's exactly how it was constructed,
see~\eqref{eq:Tt}. Recall that the core $\calU$ consists of finite
linear combinations of eigenfunctions $\varphi_j$. Equivalently, we might have
considered $\calU^\prime = \sfC_0^2[0,\infty)$. Indeed, if $\chi_0$ is a smooth
function which is $1$ on $(-\infty, 0]$ and $0$ on $[1, \infty)$ and if
$\chi_R (r) = \chi_0 (r-R)$, then, for any $j$,
\be{eq:chi-approx}
\lim_{R\to\infty} \chi_R\varphi_j = \varphi_j
\quad\text{and}\quad
\lim_{R\to\infty}\sfL_{\sigma,q}(\chi_R\varphi_j) = \sfL_{\sigma,q}\varphi_j
= -\eig_j\varphi_j .
\ee
Above, both convergences are pointwise and in $\bbL_2(\bbR_+)$. In order to
check the second claim in~\eqref{eq:chi-approx}, just note that
\[
\sfL_{\sigma,q}(\chi_R\varphi_j) = -\eig_j\varphi_j +
\frac{\sigma^2}{2} (\varphi_j\chi_R^{\prime\prime}
+ 2\varphi_j^\prime\chi_R^\prime) ,
\]
and the conclusion follows, since both $\varphi_j$ and $\varphi_j^\prime$
belong to $\bbL^2$.

Consider, therefore, $u\in\sfC_0^2[0,\infty)$. Define $u_\lambda(\sfr) =
u(\sfr)$. Clearly, $\lim_{\lambda\to 0} u_\lambda = u$. On the other hand (see Remark~\ref{rem:NZ}),
\begin{multline}
\label{eq:u-to-v}
E_\lambda \mathrm{e}^{ V_\lambda (\Hla\sfr)} \sfL_\lambda u_\lambda(\sfr)
=
\frac{1}{h_\lambda^2} \Bigl(
\sum_{\sfs\in \bbZ_\lambda}  p_\lambda (\sfs-\sfr) \mathrm{e}^{ \frac{V_\lambda (\Hla\sfr) -
V_\lambda (\Hla\sfs)}{2}} u(\sfs) -
E_\lambda \mathrm{e}^{ V_\lambda (\Hla\sfr)} u_\lambda(\sfr)
\Bigr)
\\
=
\frac{1}{h_\lambda^2}
\sum_\sfs  p_\lambda (\sfs-\sfr) \bigl(
\mathrm{e}^{ \frac{V_\lambda(\Hla\sfr) - V_\lambda (\Hla\sfs)}{2}}
u(\sfs) - u(\sfr) \bigr)
+ \frac{1 - E_\lambda \mathrm{e}^{ V_\lambda (\Hla\sfr)}}{h_\lambda^2}
u(\sfr).
\end{multline}
Choose $R$ such that ${\rm supp}(u)\in [0,R]$. Possibly going to a sub-sequence  assume
that $\sfe = \lim \sfe_\lambda$ exists.
Then, by our assumptions on $V_\lambda$,
the second term converges to $(\sfe - q (\sfr))u (\sfr )$, uniformly in $r\in [0,R]$.

As for the first term in~\eqref{eq:u-to-v}, note that,
since $u_\lambda(\sfs) \equiv 0 $ for $\sfs > R$
and $p_\lambda(\sfs-\sfr) \leq \mathrm{e}^{-c \Hla \abs{\sfs-\sfr}}$,
we may restrict attention to $\sfr,\sfs \leq R+1$.  But then, again  by our
assumptions on $V_\lambda$, the quantity
\[
\abs{(V_\lambda(\Hla\sfs) - V_\lambda(\Hla\sfr))
p_\lambda(\sfs-\sfr)}
= h_\lambda^2 \abs{\lb q(\sfs) - q(\sfr) \rb p_\lambda(\sfs-\sfr)} +
\smo{h_\lambda^2} =
\smo{h_\lambda^2} .
\]
Finally, by our assumptions~\eqref{eq:Assumption} and~\eqref{eq:sigma}
on the underlying random walk,
\[
\lim_{\lambda\to 0}
\frac{1}{h_\lambda^2} \sum_{\sfs \in \bbZ_\lambda}
p_\lambda(\sfs-\cdot) (u(\sfs)-u(\cdot)) =
\frac{\sigma^2}{2} u^{\prime\prime} (\cdot ) ,
\]
in the sense of Definition~\ref{def:ConvHS}.

We have proved:
\begin{proposition}
\label{thm:ConvSG}
Under our assumptions on $V_\lambda$ and $\sfp$, the following convergence (in
the sense of Definition~\ref{def:ConvHS}), holds uniformly in $t$ on compact
subsets of $\bbR_+$:
If $\lim_{k\to\infty} \sfe_{\lambda_k} = \sfe$ and
$\lim_{k\to\infty} f_{\lambda_k} = f$, then
\be{eq:ConvSG}
\lim_{k \to \infty}
\sfT_{\lambda_k}^{\lfloor H_{\lambda_k}^2 t\rfloor} f_{\lambda_k} =
\mathrm{e}^{(\sfL +\sfe\sfI)t} f .
\ee
\end{proposition}

\paragraph{Convergence of eigenvalues and eigenfunctions}
\begin{theorem}
\label{thm:phihat}
Under our assumptions on $V_\lambda$ and $\sfp$,
\be{eq:phihat}
\eig_0 = \lim_{\lambda\to 0}\sfe_\lambda , \quad
\varphi_0 = \lim_{\lambda\to 0}\phi_\lambda = \lim_{\lambda\to 0}\phi_\lambda^*
\quad\text{and}\quad
\lim_{\lambda\to 0}\frac{c_\lambda}{h_\lambda} = 1 .
\ee
\end{theorem}
\begin{proof}
By Lemma~\ref{lem:elambda}, the set $\{\sfe_\lambda\}$ is bounded and,
by Proposition~\ref{prop:Comp}, the set $\{\phi_\lambda\}$ is sequentially
compact.
Let $\lambda_k\searrow 0$ be a sequence such that both
$\sfe = \lim_{k\to\infty} \sfe_{\lambda_k}$ and
$\varphi = \lim_{k\to\infty}\phi_{\lambda_k}$ exist.
Then Proposition~\ref{thm:ConvSG} implies that
\[
\varphi = \mathrm{e}^{(\sfL +\sfe\sfI)t}\varphi .
\]
By compactness, $\|\varphi\|_2 = 1$.
In other words, $\varphi$ is a non-negative normalized
$\bbL_2(\bbR_+)$-eigenfunction of $\sfL$ with eigenvalue $-\sfe$. Which means
that $\varphi = \varphi_0$ and $\sfe = \eig_0$. Exactly the same argument
applies to $\phi_\lambda^*$.

By construction (see~\eqref{eq:c-lambda}),
$1\equiv
c_\lambda \sum_{\sfr\in\bbN_\lambda} \phi_\lambda(\sfr) \phi_\lambda^*(\sfr) =
\frac{c_\lambda}{h_\lambda}
\langle\phi_\lambda,\phi_\lambda^*\rangle_{2,\lambda}$.
Since, by the second assertion of~\eqref{eq:phihat},
$\lim_{\lambda\to 0} \langle \phi_\lambda,\phi_\lambda^* \rangle_{2,\lambda} =
\|\varphi_0\|_2^2 = 1$, the last claim of Theorem~\ref{thm:phihat} follows
as well.
\qed
\end{proof}

\paragraph{Convergence of finite-dimensional distributions}
Recall our notations $\bbP_\lambda$ and $\bbP_{\sigma,q}$ for the path
measures of the ground-state chain $X_n$ and the Ferrari-Spohn diffusion
$x(t)$. Recall also our rescaling of the ground-state chain:
$x_\lambda(t) = h_\lambda X_{\lfloor \Hla^2 t\rfloor}$.
\begin{corollary}
\label{thm:ConvFDD}
For any $k$, any $0 < s_1 < s_2 <\dots <s_k$ and for any collection of bounded
continuous functions  $u_0,\ldots,u_k\in\sfC_{\rm b}(\bbR_+)$,
\begin{multline}\label{eq:ConvFDD}
\lim_{\lambda\to 0}
\bbE_\lambda\{u_0(x_\lambda(0)) u_1(x_\lambda(s_1))\cdots
u_k(x_\lambda(s_k))\}\\
= \bbE_{\sigma,q} \{ u_0(x(0)) u_1(x(s_1))\cdots u_k(x(s_k))\} .
\end{multline}
\end{corollary}
\begin{proof}
Set $s_0=0$ and $t_i = s_i - s_{i-1}$.
Since $\mu_\lambda = {c_\lambda} \phi_\lambda^*\phi_\lambda$ and in view of the
expressions~\eqref{eq:pi-lambda} for transition probabilities $\pi_\lambda$ of
the ground-state chain, the rightmost asymptotic relation in~\eqref{eq:phihat},
and~\eqref{eq:T-FS} for Ferrari-Spohn semigroups, the target
formula~\eqref{eq:ConvFDD} can be written as
\begin{multline}
\label{eq:ConvFDD-1}
\lim_{\lambda\to 0}
h_\lambda \sum_{\sfr\in\bbN_\lambda} \phi^*_\lambda (\sfr) u_{\lambda,0}(\sfr)
\sfT_\lambda^{\lfloor \Hla^2 t_1\rfloor}
\bigl( u_{\lambda,1}\sfT_\lambda^{\lfloor \Hla^2 t_2\rfloor}\bigl(
u_{\lambda,2}\cdots\sfT_\lambda^{\lfloor \Hla^2 t_k\rfloor}
(u_{\lambda,k}\phi_\lambda ) \cdots\bigr)\bigr) (\sfr)\\
=
\int_0^\infty \varphi_0 (r) u_0(r) \sfT^{t_1}\bigl( u_1\sfT^{t_2} \bigl(
u_2\cdots \sfT^{t_k} (u_k\varphi_0)\cdots \bigr)\bigr) (r) \,\dd r,
\end{multline}
where $u_{\lambda,i}$ and $u_i$ coincide on $\bbN_\lambda$.
Theorem~\ref{thm:phihat} implies that
$\lim_{\lambda\to 0}\phi_\lambda = \varphi_0$ and
$\lim_{\lambda\to 0}\phi_\lambda^* = \varphi_0$.
Hence, by induction, \eqref{eq:ConvFDD-1} is a consequence of
Proposition~\ref{thm:ConvSG} and the following two elementary facts:
\begin{enumerate}[(a)]
\item If $\lim v_\lambda = v$ and $u_\lambda(\sfr) = u(\sfr)$ with $u$ being a
bounded continuous function, then $\lim_{\lambda\to 0} v_{\lambda} u_\lambda =
vu$.
\item If $\lim u_\lambda =u$ and $\lim v_\lambda =v$, then
$\lim_{\lambda\to 0} \langle  u_\lambda , v_\lambda \rangle_{2,\lambda} =
\langle u, v\rangle_2$.\qed
\end{enumerate}
\end{proof}

\section{Probabilistic tools}
\label{sec:tools}
The derivations of the probabilistic estimates given below are based on the
techniques and ideas developed in~\cite{HV04}. Nevertheless, because our
setting is slightly different and for completeness, we provide detailed proofs.
In addition, one of the needed claims from~\cite{HV04} (Theorem~1.2 therein)
contains a mistake, which we correct here.

\medskip
Recall our notation $\hat\calP_{N,+}^{\sfu,\sfv} = \calP_{1,N,+}^{\sfu,\sfv}$.
As before, given a path $\bbX = (X_1,\ldots,X_N)$, set
$\RWP(\bbX)=\prod_{i=1}^{N-1} p_{X_{i+1}-X_i}$.
Define $\hat\calP_{N,+}^{\sfu,\varnothing} =
\cup_{\sfv\in\bbZ_+} \hat\calP_{N,+}^{\sfu,\sfv}$ and consider the partition
functions
\[
\hat Z_{N,+,\lambda}^{\sfu,\varnothing}
=
\sum_{\bbX\in\hat\calP_{N,+}^{\sfu,\varnothing}}
e^{-\sum_{i=1}^N V_\lambda (X_i)} \, \RWP(\bbX).
\]
More generally, given any subset
$\calC\subset\hat\calP_{N,+}^{\sfu,\varnothing}$, we denote by
\[
\hat Z^{\sfu,\varnothing}_{N,+,\lambda}[\calC]
=
\sum_{\bbX\in\calC}
e^{-\sum_{i=1}^N V_\lambda (X_i)} \, \RWP(\bbX),
\]
the partition function restricted to paths satisfying the constraint $\calC$.

\subsection{Proof of Lemma~\ref{lem:elambda}}
\label{sub:elambda}
The proof will rely on the following identity
which, exactly  as  \eqref{eq:ReprTlambda}, is straightforward from the very 
definition of $\tilde\sfT_\lambda$ in \eqref{eq:tTlambda}:
\[
\hat Z_{N,+,\lambda}^{\sfu,\varnothing} =
\mathrm{e}^{-\frac{1}{2} V_\lambda (\sfu)} \tilde\sfT_\lambda^{N}[f_\lambda](\sfu),
\]
where $f_\lambda (x)  = \mathrm{e}^{-\frac{1}{2}V_\lambda (x)}$.
Since $f_\lambda$ is positive,
\be{eq:El-ZN}
\log E_\lambda =
\lim_{N\to\infty} \frac1N \log\hat Z_{N,+,\lambda}^{\sfu,\varnothing}\,,
\ee
for all $\sfu\in\bbZ^+$.
In particular, the claim of Lemma~\ref{lem:elambda} will follow from lower and
upper bounds on $\hat Z_{N,+,\lambda}^{\sfu,\varnothing}$ for finite values of
$N$ and $\lambda$.
In the sequel, we shall allow rather general values of the boundary condition
$\sfu$. Of course, to derive the claim of Lemma~\ref{lem:elambda}, we could as
well take $\sfu=0$.

We shall compare the tilted partition functions
$\hat Z_{N,+,\lambda}^{\sfu,\varnothing}$
and $\hat Z_{N,+,0}^{\sfu,\varnothing}$. The latter equals to
the probability that the random walk starting at $\sfu$ stays positive
for first $N$ steps of its life. This probability is evidently non-decreasing
with $\sfu$ and, as is well known (see for instance~\cite{AD}),
it is of order $N^{-1}$ for $\sfu = 1$.
In particular,
\be{eq:pf-not}
\lim_{N\to\infty}\frac{1}{N} \log \hat Z_{N,+,0}^{\sfu,\varnothing} = 0,
\ee
uniformly in $\sfu\in\bbN$.

\paragraph{Lower bound on $\hat Z^{\sfu,\varnothing}_{N,+,\lambda}$ and upper
bound on $\sfe_\lambda$}

\begin{figure}
\begin{center}
\scalebox{.43}{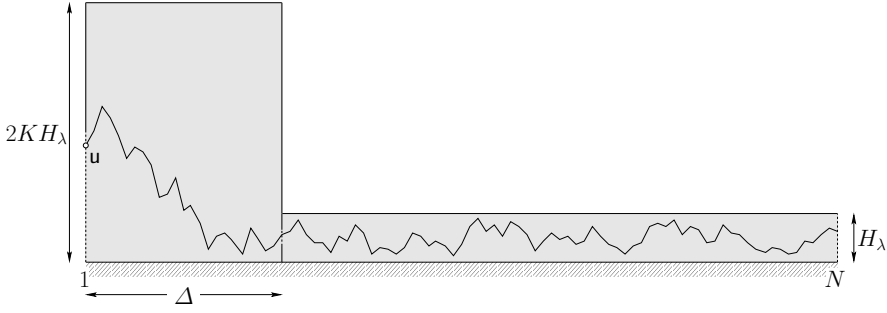}
\end{center}
\caption{The construction for the lower bound in the proof of
Lemma~\ref{lem:elambda}.}
\label{fig:LB_Z}
\end{figure}

We claim that there exist finite constants $\overline{\sfe}$ and $c_1$ such
that, for any $K \geq 1$ fixed,
\begin{equation}
\label{eq_lowerbound_Z}
\hat Z^{\sfu,\varnothing}_{N,+,\lambda}
\geq
e^{-\overline{\sfe}  N \Hla^{-2} -c_1 K\sqrt{q(2K)}}\,
\hat Z^{\sfu,\varnothing}_{N,+,0},
\end{equation}
uniformly in $\lambda$ small,  $0\leq\sfu\leq K\Hla$ and
\be{eq:Delta}
N \gg  \Delta = \Delta(K,\lambda) = \frac{K\Hla^2}{\sqrt{q(2K)}} .
\ee
In view of~\eqref{eq:El-ZN} and~\eqref{eq:pf-not}, this implies that
$\sfe_\lambda = - \Hla^2 \log E_\lambda \leq \overline{\sfe}$ for all
$\lambda$ sufficiently small.

In order to check~\eqref{eq_lowerbound_Z}, we restrict the partition function
to trajectories made of two pieces (see Figure~\ref{fig:LB_Z}). 
The left part
is used to bring the interface below $\Hla$; in the remaining piece, the
interface remains inside a tube of height $\Hla$.

\smallskip
We consider the events
\footnote{Here and several times in the sequel, we assume numbers like $\Delta$
to be integers whenever it is desirable.}
\begin{gather*}
\calD_L = \bsetof{\bbX\in\hat \calP_{N,+}^{\sfu,\varnothing}}
{\max_{i\in \{1,\ldots,\Delta\}} X_i \leq 2K\Hla,
X_\Delta\in [\tfrac13\Hla,\tfrac23\Hla]},\\
\calD_M = \bsetof{\bbX\in\hat \calP_{N,+}^{\sfu,\varnothing}}
{\max_{i\in \{\Delta,\ldots,N\}} X_i \leq \Hla}.
\end{gather*}
Then
\[
\hat Z^{\sfu,\varnothing}_{N,+,\lambda}
\geq
e^{-\Delta V_\lambda(2K\Hla) - (N-\Delta) V_\lambda(\Hla)}\,
\RWP(\calD_L\cap\calD_M \given \hat\calP^{\sfu,\varnothing}_{N,+})\,
\hat Z^{\sfu,\varnothing}_{N,+,0}.
\]
By the assumptions~\eqref{eq:HL-1} and~\eqref{eq:HL-2},
$H_\lambda^2 V_\lambda (\sfr \Hla) < 2 q (\sfr)$ uniformly in
$\sfr\in [0,2K]$, for all $\lambda$ sufficiently small.
Hence, for such
$\lambda$, the exponent in the right-hand side is bounded below by
$e^{-2 K\sqrt{q(2K)} - 2 N\Hla^{-2}}$.

It remains to estimate
$\RWP(\calD_L\cap\calD_M \given\calP^{\sfu,\varnothing}_{N,+})$.
By the invariance principle (for a random walk conditioned to stay
positive; see first \cite[Theorem~1]{BD94} and then~\cite[Theorem~1.1]{CCh08}),
\[
\liminf_{N\to\infty} \,
\RWP(\calD_L\given \calP^{\sfu,\varnothing}_{N,+})
\geq
e^{-c_2 K\sqrt{q(2 K)}},
\]
for some absolute constant $c_2>0$, provided that $\lambda$ be small enough.

On the other hand, letting $\calD' = \{\sup_{0\leq i\leq \Hla^2}
X_i \leq
\Hla\}\cap\{X_{\lceil\Hla^2\rceil}\in[\tfrac13\Hla,\tfrac23\Hla]\}$, it follows
from the Markov property that
\begin{align*}
\inf_{\ell\in[\tfrac13\Hla,\tfrac23\Hla]}
\RWP(\calD_M \given &X_L=\ell, X_i\geq 0\;\forall \Delta\leq i \leq N)\\
&\geq
\Bigl\{
\inf_{\ell\in[\tfrac13\Hla,\tfrac23\Hla]}
\RWP(\calD' \given X_0=\ell, X_i\geq
0\;\forall \Delta\leq i \leq \Hla^2)
\Bigr\}^{\lceil N/\Hla^2\rceil}\\
&\geq
e^{-c_3 N \Hla^{-2}} .
\end{align*}

\paragraph{Upper bound on $\hat Z^{\sfu,\varnothing}_{N,+,\lambda}$ and
lower bound on $\sfe_\lambda$}

\begin{figure}
\begin{center}
\scalebox{.5}{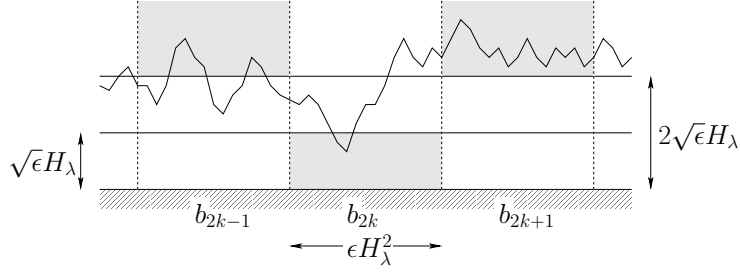}
\end{center}
\caption{The event $\calB_k$ occurs if the path visits both leftmost and
rightmost shaded areas. The event $\calC_k$ occurs if, in addition, it also
visits the third one.}
\label{fig:UB_Z}
\end{figure}

We claim that there exist $\bar\lambda>0 $ and a positive constant
$\underline{\sfe}$ such that
\begin{equation}
\label{eq_upperbound1_Z}
\hat Z^{\sfu,\varnothing}_{N,+,\lambda} \leq
e^{-\underline{\sfe} N \Hla^{-2}} \hat Z^{\sfu,\varnothing}_{N,+,0},
\end{equation}
uniformly in $\sfu\geq 0$, $N\geq\Hla^2$ and $\lambda<\bar\lambda$.

\smallskip
Let us fix some small $\epsilon>0$ (which does not have to be very small; one
can optimize over it at the end of the proof).
The idea behind the proof is that a typical trajectory has many disjoint
segments of the length at least $\epsilon H_{\lambda }^{2},$ which
are at a distance at least $\sqrt{\epsilon }\Hla$ from the wall.

We partition  the interval $\{1,\ldots,N\}$ into $N_\lambda$ disjoint
intervals $b_1,\ldots,b_{N_\lambda}$ of length $\epsilon\Hla^2$ and, possibly,
one additional shorter rightmost interval.

We say that the event $\calB_k$ occurs if (see Figure~\ref{fig:UB_Z})
\begin{equation}\label{eq_upperbound_Z}
\max_{i\in b_{2k-1}} X_i > 2\sqrt{\epsilon}\Hla \qquad\text{and}\qquad
\max_{i\in b_{2k+1}} X_i > 2\sqrt{\epsilon}\Hla.
\end{equation}
Let us denote by $G$ the number of indices $k$ for which the event
$\calB_k$ occurs.
It follows from the CLT that there exists $\kappa_1 >0$ such that
\be{eq:bound-uw}
\inf_{\sfv,\sfw\geq 0}
\RWP\bigl(\max_{1\leq i\leq \epsilon \Hla^2} X_i >
2\sqrt{\epsilon}\Hla \,\bigm|\, \calP_{1,\epsilon\Hla^2,+}^{\sfv,\sfw}\bigr) >
\kappa_1 .
\ee
Observe that the events
$\{\max_{i\in b_{2j-1}} X_i > 2\sqrt{\epsilon}\Hla\}$,
$j=1,\ldots,N_\lambda/2$, are conditionally independent
given the trajectories in the intervals $b_{2k}$.
As a result, \eqref{eq:bound-uw} implies that there exists $\kappa_2>0 $ such
that
\be{eq:Gbound-1}
\RWP(G \leq \tfrac18 \kappa_1^2 N_\lambda \given
\hat\calP_{N,+}^{\sfu,\varnothing}) \leq e^{-\kappa_2  N_\lambda} ,
\ee
uniformly in $\sfu$.

Similarly, let us say that the event $\calC_k$ occurs if $\calB_k$ occurs and
(see Figure~\ref{fig:UB_Z})
\[
\min_{i\in b_{2k}} X_i < \sqrt{\epsilon}\Hla ,
\]
and let us denote by $G'$ the number of indices such that $\calC_k$ occurs.

The occurrence of $\calC_k$ enforces a downward fluctuation at least as large as
$\sqrt{\epsilon}\Hla$ on a time interval of length at most $3\epsilon\Hla^2$.
The functional CLT~\cite[Theorem~2.4]{CCh13} implies that such an event has
probability at most
$\kappa_3$, for some $\kappa_3<1$, uniformly in $\lambda$ small.
This implies that there exists $\kappa_4 >0$, such that
\be{eq:Gbound-2}
\RWP\bigl( G' \geq \frac{1+\kappa_3}{2} g \,\bigm|\,
\hat\calP_{N,+}^{\sfu,\varnothing} ; G=g\bigr) \leq
\mathrm{e}^{-\kappa_4 g},
\ee
uniformly in $\sfu$ and $g$. Altogether, \eqref{eq:Gbound-1} and
\eqref{eq:Gbound-2} yield
\be{eq:Gbound-3}
\RWP\bigl( G- G' \leq \frac{1-\kappa_3}{8} \kappa_1^2 N_\lambda
\,\bigm|\, \hat\calP_{N,+}^{\sfu,\varnothing}\bigr)
\leq \mathrm{e}^{-\kappa_5 N_\lambda}
\leq \mathrm{e}^{-\kappa_6 \epsilon^{-1} \Hla^{-2}N} .
\ee
The quantity $G-G'$ provides a lower bound on the number of disjoint intervals
$b_{2k}$ of length $\epsilon \Hla^2$ such that
$\min_{i\in b_{2k}} X_i \geq \sqrt{\epsilon} \Hla$. Therefore,
\be{eq:Gbound-4}
\sum_{i=1}^N V_\lambda(X_i) \geq
(G-G') \epsilon \Hla^{2} V(\sqrt{\epsilon}\Hla)
\geq (G-G') \epsilon q_0(\sqrt{\epsilon})
\geq \kappa_7 q_0(\sqrt{\epsilon}) \Hla^{-2} N ,
\ee
whenever $G-G' \geq \kappa_7 N_\lambda$. Take
$\kappa_7 = \frac{1-\kappa_3}{8}\kappa_1^2$.
The conclusion~\eqref{eq_upperbound1_Z} follows from~\eqref{eq:Gbound-3}
and~\eqref{eq:Gbound-4}.
\qed

\subsection{Proof of Lemma~\ref{lem:phi-l-K}}
\label{sub:lem-phi}
We shall prove Lemma~\ref{lem:phi-l-K} only for $\phi_\lambda^*$. The proof for
$\phi_\lambda$ is a literal repetition for reversed walks.
For the sake of notations, we shall think of $\sfT_\lambda$
in~\eqref{eq:Tlambda} as acting on non-rescaled spaces $\ell_2(\bbN)$,
with the norm
\[
\langle u, v\rangle_{2,\lambda} = h_\lambda \sum_{\sfr\in\bbN} u(\sfr)v(\sfr) .
\]
Compare with~\eqref{eq:l2-L}.

Similarly, we shall think of $\phi_\lambda$ and $\phi_\lambda^*$ as of
functions on $\bbN$.
Recall the normalizing constant $c_\lambda$ which was introduced 
in~\eqref{eq:c-lambda}, and recall that $\mu_\lambda (\sfx ) = c_\lambda 
\phi_\lambda (\sfx ) \phi_\lambda^* (\sfx )$ is the invariant measure of the 
positively recurrent chain on $\bbN$ with transition probabilities 
$\pi_\lambda$ specified in
\eqref{eq:pi-lambda}.
Define $g_M (\sfx ) = \1_{\{x > M\}} = \sum_{\sfx >M}\1_\sfx$.
Then,
\be{eq:T-l-lim}
\begin{split}
\lim_{N\to\infty} \sfT_\lambda^N g_M (\cdot) &=
\phi_\lambda (\cdot )
\lim_{N\to\infty} \pi_\lambda^N \left[
\sum_{\sfx>M}\frac{\1_\sfx}{\phi_\lambda (\sfx )}
\right]
(\cdot) = \phi_\lambda (\cdot )\sum_{\sfx >M}\frac{\mu_\lambda 
(\sfx)}{\phi_\lambda (\sfx )}
\\
&=
c_\lambda \phi_\lambda(\cdot) \sum_{\sfx>M} \phi_\lambda^*(\sfx) ,
\end{split}
\ee
for any $\lambda>0$.

For $\sfv>M$ and $k\geq 0$ let $\calQ_{k,+}^{\sfv ,M}$ be the family of $k$-step
paths $\bbX = (\sfx_0, \ldots, \sfx_k)$
which start at $\sfv$, $\sfx_0 = \sfv$, stay above level $M$,
and end up above level $2M$,  $\sfx_k>2M$. We employ the notation
(see~\eqref{eq:Tlambda})
\be{eq:restr-PF}
\sfT_\lambda^k \{\calQ_{k,+}^{\sfv ,M}\} =
\sum_{\bbX\in\calQ_{k,+}^{\sfv ,M}} \prod_1^k \sfT_\lambda (\sfx_{i-1},\sfx_i ) .
\ee
By convention, $\sfT_\lambda^0 \{\calQ_{0,+}^{\sfv ,M}\} = \1_{\{\sfv >2M\}}$.

\begin{figure}
\begin{center}
\scalebox{.46}{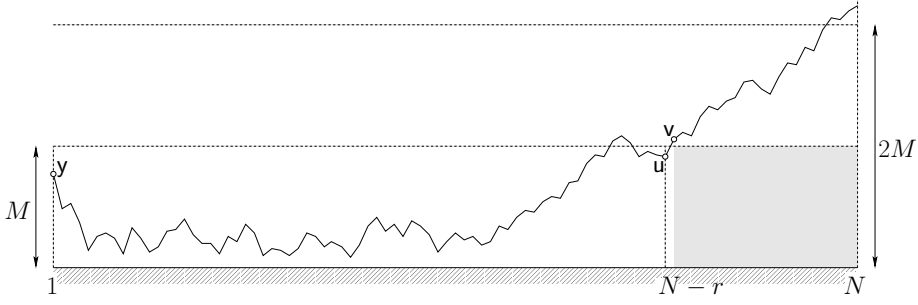}
\end{center}
\caption{The last exit decomposition in~\eqref{eq:last-exit}. After time $N-r$,
the path cannot visit the shaded area and has to end up above level $2M$.}
\label{fig:LastExit}
\end{figure}

Let us fix $0<\sfy<M $ and consider paths
$\bbX\in\hat\calP_{N,+}^{\sfy,\emptyset}$
ending up above level $2M$, $X_N > 2M$.
By the last exit decomposition from $\{1,\ldots,M \}$ (see
Figure~\ref{fig:LastExit}),
\be{eq:last-exit}
\begin{split}
\sfT_\lambda^N g_{2M} (\sfy) &=
\sum_{\sfu \leq M} \sum_{\sfv > M}
\sum_{r=1}^{N-1} \sfT_\lambda^{N-r-1}
\left[\1_\sfu \right] (\sfy) \sfT_\lambda (\sfu ,\sfv)
\sfT_\lambda^{r-1}\{\calQ_{r-1,+}^{\sfv, M}\}
\\
&=
\sum_{\sfu \leq M} \sum_{\sfv > M}
\sum_{r=1}^{N-1} \sfT_\lambda^{N-r-1} \left[\1_\sfu \right] (\sfy)
E_\lambda^{-1} p_{\sfv-\sfu} \mathrm{e}^{-\frac{ V_\lambda(\sfu) +
V_\lambda(\sfv)}{2}} \sfT_\lambda^{r-1}\{\calQ_{r-1,+}^{\sfv, M}\} .
\end{split}
\ee
Taking the limit $N\to\infty$, we infer from~\eqref{eq:T-l-lim},
\eqref{eq:last-exit} and positivity of $\phi_\lambda$ and $c_\lambda$
that
\be{eq:leq-geq}
\sum_{\sfx > 2M} \phi^*_\lambda(\sfx) =
\sum_{\sfu \leq M} \phi^*_\lambda(\sfu) \sum_{\sfv > M} E_\lambda^{-1}
p_{\sfv-\sfu} \mathrm{e}^{-\frac{ V_\lambda(\sfu) + V_\lambda(\sfv)}{2}}
\sum_{r\geq 0} \sfT_\lambda^{r} \{\calQ_{r,+}^{\sfv ,M}\} .
\ee
Let us try to derive an upper bound on
\be{eq:max-term}
\max_{\sfu\leq M} \sum_{\sfv > M} E_\lambda^{-1} p_{\sfv-\sfu}
\mathrm{e}^{-\frac{ V_\lambda(\sfu) + V_\lambda(\sfv)}{2}}
\sum_{r\geq 0} \sfT_\lambda^{r}\{\calQ_{r,+}^{\sfv ,M}\} .
\ee

For simplicity, we shall prove directly the second inequality 
in~\eqref{eq:tail-phi-l}.
The arguments rely on the lower bound $H_\lambda^2 V_\lambda (M) \geq
q_0 (h_\lambda M)$. If instead we keep track of the original quantity
$H_\lambda^2 V_\lambda (M)$, then the first inequality in~\eqref{eq:tail-phi-l} 
will follow.

By Lemma~\ref{lem:elambda} and by our assumptions on $\sfp$ and $V_\lambda$,
\be{eq:v-bound-1}
E_\lambda^{-1} p_{\sfv-\sfu}
\mathrm{e}^{-\frac{ V_\lambda(\sfu) + V_\lambda(\sfv)}{2}}
\leq
\exp\bigl\{
- c_1 h_\lambda^2 \bigl( q_0 (h_\lambda M) - 1\bigr) - c_2 (\sfv - M )\bigr\} .
\ee
On the other hand, Lemma~\ref{lem:elambda} and crude estimates on the values
of the potential $V$ above level $M$ and of the hitting probability of
the half-line $\{2M,2M+1,\ldots\}$ by an $r$-step random walk which starts at
$\sfv$ imply that, for $r>1$,
\begin{align}
\sfT_\lambda^{r}\{\calQ_{r,+}^{\sfv ,M}\}
&\leq
E_\lambda^{-r} \mathrm{e}^{-r h_\lambda^2 q_0 (h_\lambda M)}
\RWP(X_r > 2M-\sfv) \notag\\
&\leq
\exp\Bigl\{ - c_3 h_\lambda^2 r \bigl( q_0(h_\lambda M) - 1\bigr) -
c_4 \frac{(2M-\sfv)_+^2}{r}\wedge (2M-\sfv)_+\Bigr\} .
\label{eq:Up-term-bound}
\end{align}
Indeed, the second term in the exponent on the right-hand side above follows
from the exponential Markov inequality $\RWP(X_r > a) \leq
\mathrm{e}^{-c_4\frac{a_+^2}{r}\wedge a_+}$.

The right-hand sides of both~\eqref{eq:v-bound-1} and~\eqref{eq:Up-term-bound}
are already independent of $\sfu$. Let us sum over $\sfv > M$. If $r=0$, then
$\sfv$ has to satisfy $\sfv>2M$.
Using~\eqref{eq:v-bound-1},
\be{eq:ubound_r0}
\sum_{\sfv>2M} E_\lambda^{-1} p_{\sfv-\sfu}
\mathrm{e}^{-\frac{ V_\lambda(\sfu) + V_\lambda(\sfv)}{2}} \leq
\exp\bigl\{-c_5\bigl( h_\lambda^2\bigl( q_0(h_\lambda M)-1 \bigr) + M
\bigr)\bigr\} .
\ee
For $r>0$, we take advantage of both upper bounds~\eqref{eq:v-bound-1}
and~\eqref{eq:Up-term-bound} above:
\be{eq:second-bound-v}
\sum_{\sfv>M} \mathrm{e}^{- c_2 (\sfv-M) - c_4 \frac{(2M-\sfv)_+^2}{r}
\wedge (2M-\sfv)_+} \leq \mathrm{e}^{-c_6 \frac{M^2}{r}\wedge M} .
\ee
Putting things together for $M=\Hla K$, we conclude that the expression
in~\eqref{eq:max-term} is bounded above by
\be{eq:FinBound-max}
\mathrm{e}^{-c_7 \Hla K} +
\sum_{r \geq 1} \exp\Bigl\{ - c_8 \bigl( h_\lambda^2 r q_0(K) +
\frac{(\Hla K)^2}{r}\wedge \Hla K \bigr)\Bigr\}
\leq c_{10} \mathrm{e}^{-c_9 K(\sqrt{q_0(K)} \wedge \Hla)} ,
\ee
uniformly in $K>0$ and $\lambda$ sufficiently small.

Coming back to~\eqref{eq:leq-geq}, we infer: For any $K>0$ fixed,
\be{eq:leq-geq-fin}
\sum_{\sfx > \Hla K} \phi^*_\lambda(\sfx) \leq
c_{11} \mathrm{e}^{-c_{12} K \sqrt{q_0(K)}}
\sum_\sfx \phi^*_\lambda(\sfx) ,
\ee
for all $\lambda <\lambda_0(K)$.
Notice that $\sum_\sfx \phi^*_\lambda(\sfx) <\infty$ by compactness of
$\sfT_\lambda$.

\smallskip
Let us return to our basic rescaling~\eqref{eq:Normalization} of $\phi_\lambda$
and $\phi_\lambda^*$ as unit norm elements of $\ell_2(\bbN_\lambda)$. The
bound~\eqref{eq:leq-geq-fin} can be rewritten as
\be{eq:tail-phi-star}
h_\lambda \sum_{\sfr > K} \phi^*_\lambda(\sfr) \leq
c_{11} \mathrm{e}^{-c_{12} K\sqrt{q_0(K)}} \|\phi_\lambda^*\|_{1,\lambda} .
\ee
Since $h_\lambda \sum_{\sfr\leq K} \phi^*_\lambda(\sfr) \leq
\sqrt{K} \|\phi_\lambda^*\|_{2,\lambda} = \sqrt{K}$, we conclude that
$\{\| \phi_\lambda^* \|_{1,\lambda}\}$ is a bounded sequence. The
bound~\eqref{eq:ell-1-bound} and, in view of~\eqref{eq:tail-phi-star},
also~\eqref{eq:tail-phi-l} follow. \qed

\subsection{Tightness of $(x_\lambda,\bbP_\lambda)$}
\label{sub:tightness}
Fix any $T <\infty$ and consider the family of rescaled processes $x_\lambda$
defined in~\eqref{eq:xN}.
Precisely, $x_\lambda$ is the linear interpolation of the rescaled  stationary
ergodic ground-state chain $X= X^\lambda $ with transition probabilities $\pi_\lambda$
defined in~\eqref{eq:pi-lambda} and invariant distribution $\mu_\lambda$.
With a slight abuse of notation, we shall continue to use $\bbP_\lambda$ for
the induced distribution of $x_\lambda(\cdot)$ on $\sfC[-T,T]$.
\begin{proposition}
\label{prop:tightness}
The family $(x_\lambda,\bbP_\lambda)$ is tight on $\sfC[-T,T]$.
\end{proposition}

\paragraph{Proof of Proposition~\ref{prop:tightness}}
Recall that the invariant measure $\mu_\lambda$  at $\lambda>0$ is given by
$\mu_\lambda = c_\lambda\phi_\lambda\phi_\lambda^*$. Thus, by
Theorem~\ref{thm:phihat}, the sequence $\{x_\lambda(0)\}$ is tight. It
remains to show that, for each $\epsilon,\nu>0$, there exists $\delta>0$ such
that
\be{eq:var-estimate}
\bbP_\lambda\bigl(
\max_{0\leq t\leq\delta}\abs{ x_\lambda(t) - x_\lambda(0)} > \epsilon
\bigr) \leq \nu \delta.
\ee
For any event $\calA\in \sigma(X_n\, :\,  0\leq n \leq \delta\Hla^2)$,
and for any $\sfx,\sfy\in\bbN$, let us define
$\calA_{\sfx,\sfy} =
\setof{\bbX \subset \calA}{X_0=\sfx, X_{\delta\Hla^{2}}=\sfy}$.
As in~\eqref{eq:restr-PF}, we employ the notation
\be{eq:restr-PF-A}
\sfT_\lambda^{\delta\Hla^2}\{\calA_{\sfx,\sfy}\}
= \sum_{\bbX\in\calA_{\sfx,\sfy}} \prod_1^{\delta\Hla^2} \sfT_\lambda (\sfx_{i-1},\sfx_i ) .
\ee
for the corresponding restricted partition function. In this way,
\be{eq:A-bound}
\bbP_\lambda(\calA) = c_\lambda \sum_{\sfx,\sfy\in\bbN} \phi_\lambda^*(\sfx)
\sfT_\lambda^{\delta\Hla^2}\{\calA_{\sfx,\sfy}\} \phi_\lambda(\sfy) .
\ee
Since $V_\lambda\geq 0$,
\[
\sfT_\lambda^{\delta\Hla^2}\{\calA_{\sfx,\sfy}\} \leq
\mathrm{e}^{-\delta\Hla^2 \log E_\lambda}
\RWP(\calA_{\sfx,\sfy}) .
\]
By Lemma~\ref{lem:elambda},
$\{ -\sfe_\lambda = \Hla^2 \log E_\lambda\}$ is a bounded sequence.
In the case of
\[
\calA = \bigl\{ \max_{0\leq n \leq\delta\Hla^2} \abs{X_n - X_0} >
\epsilon\Hla \bigr\} ,
\]
the upper bound on the probabilities
\be{eq:sfp-bounds}
\RWP(\calA_{\sfx,\sfy}) \leq
\kappa_1 \frac{h_\lambda}{\sqrt{\delta}}
\mathrm{e}^{-\kappa_2 \frac{\epsilon^2}{\delta}\wedge (\epsilon\Hla)}
\ee
holds uniformly in $\delta,\epsilon>0$, $\sfx,\sfy\in\bbN$ and $\lambda$ small.
By Theorem~\ref{thm:phihat}, $c_\lambda/h_\lambda $ is bounded. Putting things
together, we infer that
\begin{align}
\frac{1}{\delta} \bbP_\lambda\bigl( \max_{0\leq t\leq\delta}
\abs{ x_\lambda(t) - x_\lambda(0)}>\epsilon \bigr)
&\leq
\kappa_3\mathrm{e}^{-\kappa_2\frac{\epsilon^2}{\delta}\wedge (\epsilon\Hla)
-\frac{3}{2}\log\delta}
h_\lambda^2 \sum_{\sfx,\sfy\in\bbN} \phi_\lambda^*(\sfx) \phi_\lambda(\sfy)
\notag\\
&=
\kappa_3\mathrm{e}^{-\kappa_2\frac{\epsilon^2}{\delta}\wedge (\epsilon\Hla)
-\frac{3}{2}\log\delta }
\|\phi_\lambda^* \|_{1,\lambda} \|\phi_\lambda\|_{1,\lambda} ,
\label{eq:tightness-bound}
\end{align}
uniformly in $\delta,\epsilon>0$ and $\lambda$ small. Since
$ \|\phi_\lambda\|_{1,\lambda}$ and $\|\phi_\lambda^*\|_{1,\lambda}$ are
bounded and since $\Hla$ is bounded away from zero, we are home.\qed

\subsection{Asymptotic ground-state structure of
$\bbP_{N,+,\lambda}^{\sfu,\sfv}$}
\label{sub:AsympGS}

Let us fix $C>0$ and $T>1$.
For $\lambda>0$, $\sfu,\sfv\leq C \Hla$ and $N>2T\Hla^2$, we are going to
compare the restriction
$\bbP_{N,+,\lambda}^{\sfu,\sfv,T}$ of $\bbP_{N,+,\lambda}^{\sfu,\sfv}$
to the $\sigma$-algebra
\[
\calF_{\lambda,T}= \sigma(X_i \,:\, -T\Hla^2 \leq i \leq T\Hla^2)
\]
with the restriction $\bbP_{\lambda}^T$ of $\bbP_{\lambda}$ to
$\calF_{\lambda,T}$.
\begin{proposition}
\label{prop:gs-lN}
There exists $c_1>0 $ and $K=K(C,T)<\infty$ such that
\be{eq:gs-lN}
\| \bbP_{N, +, \lambda}^{\sfu,\sfv,T}-\bbP_{\lambda}^T \|_{\rm Var}
\leq 2 \mathrm{e}^{-c_1 {N}{\Hla^{-2}}} ,
\ee
uniformly in $\lambda$ small, $N>(T+K)\Hla^2$ and $\sfu,\sfv\leq C\Hla$.
Above, $\| \cdot \|_{\rm Var}$ is the total variational norm.
\end{proposition}
As an  immediate consequence, we deduce the following
\begin{corollary}
\label{prop:gs-structure}
Let $\lambda_N$ be a sequence satisfying the assumptions of
Theorem~\ref{thm:A}. Let $C,T<\infty$ be fixed and assume that the 
sequences
$\sfu_N,\sfv_N\in\bbN$ satisfy $\sfu_N,\sfv_N \leq CH_{\lambda_N}$.
Consider the sequence of processes $x_N(\cdot) = x_{\lambda_N}(\cdot)$
defined via linear interpolation from~\eqref{eq:xN}.
With a slight abuse of notation, let $\bbP_{N,+,\lambda}^{\sfu,\sfv,T}$
and $\bbP_{\lambda_N}^T$ denote the induced distributions on $\sfC[-T,T]$
of $\bbP_{N,+,\lambda_N}^{\sfu,\sfv}$ and, respectively, of the direct
ground-state chain measure $\bbP_{\lambda_N}$.
Then,
\be{eq:gs-structure}
\lim_{N\to\infty}
\| \bbP_{N,+,\lambda_N}^{\sfu,\sfv,T}-\bbP_{\lambda_N}^T \|_{\rm Var} = 0 .
\ee
\end{corollary}

\begin{proof}[Proof of Proposition~\ref{prop:gs-lN}]
We shall use a coupling argument, considering two independent copies $X^1$ and
$X^2$ of the process, with possibly different boundary conditions.

The first step is to show that we can typically find many pieces of tubes of
length $\Hla^2$ and height of order $\Hla$ inside which both paths are
confined.

\begin{figure}
\begin{center}
\scalebox{.8}{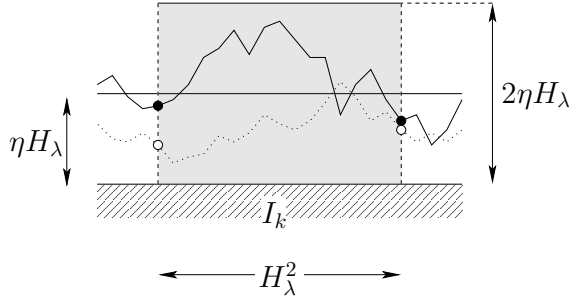}
\end{center}
\caption{The interval $I_k$ is $\eta$-good if both paths $X^1$ and $X^2$ stay
inside the shaded area and take values smaller than $\eta\Hla$ at the
boundaries of the interval.}
\label{fig:EtaGood}
\end{figure}

Let us first split the interval of length $2N+1$ into
\[
m=\lfloor (2N+1)/\Hla^2\rfloor
\]
consecutive disjoint intervals
$I_1,I_2,\ldots,I_m$, of length $\Hla^2$ (plus, possibly, a final interval of
shorter length). We say that the interval $I_k$ is \emph{$\eta$-good} if (see
Figure~\ref{fig:EtaGood})
\[
\max_{i\in I_k} X^1_i < 2\eta\Hla,\quad
\max_{i\in I_k} X^2_i < 2\eta\Hla
\]
and the values of $X_i ; i=1,2,$ at the end-points of $I_k$ are less than
$\eta\Hla$.
\begin{lemma}\label{lem_ManyEtaGoodIntervals}
Given the realizations of the two paths $X^1$ and $X^2$, let us denote by $M$
the number of $\eta$-good intervals of the form $I_{3k+2}$, $0\leq k<m/3$.
Then, there exist $\eta$, $c_2>0$, $\rho>0$ and $K_0 <\infty$ such
that
\[
\bbP_{N,+,\lambda}^{0,0}\otimes\bbP_{N,+,\lambda}^{\sfu,\sfv} (M<\rho\tfrac m3)
\leq e^{-c_2 N \Hla^{-2}} ,
\]
uniformly in $\lambda$ small, $0\leq\sfu,\sfv\leq C\Hla$ and $N\geq K_0\Hla^2$.
\end{lemma}
\begin{proof}
We first show that it is very unlikely for $X^1$ or $X^2$ to stay far away from
the wall for a long time. Indeed, let us write $B$ for the number of intervals
$I_k$ such that $\min_{i\in I_k} X_i > \eta\Hla$. Then, for any $\epsilon>0$,
there exists $\eta(C,\epsilon)$ such that for all $\eta>\eta(C,\epsilon)$,
\begin{equation}\label{eq_NotTooFarTooLong}
\bbP_{N,+,\lambda}^{\sfu,\sfv} (B>\epsilon m) \leq e^{-c_3 N \Hla^{-2}} ,
\end{equation}
for some constant $c_3>0$, uniformly in $0\leq\sfu,\sfv\leq C\Hla$. Indeed, on
the event $B>\epsilon m$,
\[
\sum_{i=-N}^N V_\lambda(X_i) \geq \epsilon m\Hla^2 V_\lambda(\eta\Hla)
\geq \epsilon q_0(\eta) (2 N\Hla^{-2} - 1) ,
\]
which provides an upper bound on $Z_{N,+,\lambda}^{\sfu,\sfv}[B>\epsilon m]$.
\begin{remark}
\label{rem:P-l-cut}
A similar argument applies for the stationary measure $\bbP_\lambda$. This
means that we may derive our target~\eqref{eq:gs-lN} for
$\bbP_\lambda ( \cdot \given X_{-N}, X_N \leq \eta \Hla )$ instead
of deriving it for $\bbP_\lambda$ itself.
\end{remark}
The claim~\eqref{eq_NotTooFarTooLong} then follows by using the lower
bound~\eqref{eq_lowerbound_Z} on the partition function (and taking $\eta$ large
enough).

\begin{figure}
\begin{center}
\scalebox{.58}{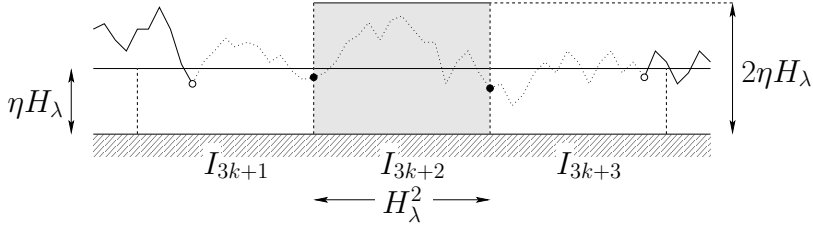}
\end{center}
\caption{When $\min_{i\in I_{3k+1}}X_i^1 < \eta\Hla$ and $\min_{i\in
I_{3k+3}}X_i^1 < \eta\Hla$, there is a uniformly (in $\lambda$) positive
probability that $\max_{i\in I_{3k+2}}X_i^1 < 2\eta\Hla$ while taking values
smaller than $\eta\Hla$ at the boundary of $I_{3k+2}$ (black dots). The white
vertices correspond to the position of the path at times $\ell_k^1$ and
$r_k^1$.}
\label{fig:PotEtaGood}
\end{figure}

Let us say that the triple $(I_{3k+1},I_{3k+2},I_{3k+3})$ is \emph{potentially
$\eta$-good} if (see Figure~\ref{fig:PotEtaGood})
\[
\min_{i\in I_{3k+1}} X^1_i < \eta\Hla,\quad
\min_{i\in I_{3k+1}} X^2_i < \eta\Hla,\quad
\min_{i\in I_{3k+3}} X^1_i < \eta\Hla,\quad
\min_{i\in I_{3k+3}} X^2_i < \eta\Hla.
\]
Let us denote by $\tilde M$  the number of potentially $\eta$-good triples.
We deduce from~\eqref{eq_NotTooFarTooLong} that, for any $\epsilon>0$, we can
find $\eta$ such that
\[
\bbP_{N,+,\lambda}^{0,0}\otimes\bbP_{N,+,\lambda}^{\sfu,\sfv}
(\tilde M \leq (1-\epsilon) \tfrac m3) \leq  e^{-c_4 N\Hla^{-2}} ,
\]
for some $c_4>0$.

Now, given a potentially $\eta$-good triple $(I_{3k+1},I_{3k+2},I_{3k+3})$, let
\[
\ell_k^1 = \min\setof{i\in I_{3k+1}}{X^1_i<\eta\Hla},\, r_k^1 =
\max\setof{i\in I_{3k+3}}{X^1_i<\eta\Hla} .
\]
Conditionally on $X^1_{\ell_k^1}$ and $X^1_{r_k^1}$, the probability that
$X^1_i\leq 2\eta\Hla$ for all $\ell_k^1<i<r_k^1$ and that both walks sit below
$\eta\Hla$ at the end-points of $I_{3k+2}$ is bounded away from zero,
uniformly in $\lambda$. Indeed, uniformly in $\sfx,\sfy<\eta\Hla$,
$n \leq 3\Hla^2$ and $1\leq k < m \leq n$,
\begin{multline*}
\bbP_{n,+,\lambda}^{\sfx,\sfy}\bigl(
\max_i X^1_i\leq2\eta\Hla ; X^1_k <\eta\Hla ; X^1_m < \eta \Hla \bigr) \\
\geq
e^{-6 q(2\eta)}\,\bbP_{n,+,0}^{\sfx,\sfy}
\bigl(\max_i X^1_i\leq 2\eta\Hla; X^1_k <\eta\Hla; X^1_m < \eta\Hla\bigr) ,
\end{multline*}
since $n V_\lambda(2\eta\Hla) \leq 6q(2\eta)$ (and $\hat
Z^{\sfx,\sfy}_{n,+,\lambda}\leq 1$).
That the latter probability is bounded below is a consequence of the invariance
principle.

The claim of the lemma now follows easily, since, conditionally on the pieces
of paths between $r_{k-1}^1$ and $\ell_k^1$, these events are independent (and
since the same argument can be made independently for $X^2$).
\qed
\end{proof}
Now that we know that we can find $O(N\Hla^{-2})$ $\eta$-good intervals, the
main observation is that, inside each such interval, there is a uniformly
positive probability that the two paths meet. Let us make this more precise:
\paragraph{Definition}
For $\frac{n}{3}\leq m\leq n$, let $\calR_{n,m}^{\sfx,\sfy}[\eta]$ be the
set of paths $\bbX = (X_1,\ldots,X_m)$ satisfying $X_1 = \sfx$, $X_m = \sfy$
and $0<X_i<2\eta\sqrt{n}, i=1,\ldots,n$.
We shall employ the short-hand notation $\calR_{n}^{\sfx,\sfy}[\eta] =
\calR_{n,n}^{\sfx,\sfy}[\eta]$;
the argument $\eta$ will often be dropped when no ambiguity arises.
We also set $\calR_{n,m}^{\sfx,\sfy,+} =
\calR_{n,m}^{\sfx,\sfy}[\infty]$.
\begin{proposition}
\label{prop:RW-estimate}
There exists $\eta_0 >0$ such that the following happens: For every $\eta \geq
\eta_0$, one can find $n_0 = n_0(\eta)\in\bbN$ and $p=p(\eta)>0$ such
that
\be{eq:RW-estimate}
\bbP_{n,+,0}^{\sfx,\sfy}\otimes\bbP_{n,+,0}^{\sfz,\sfw}
(\exists i \,:\, X^1_i = X^2_i \given
\calR_n^{\sfx,\sfy}\times\calR_n^{\sfz,\sfw}) \geq p ,
\ee
uniformly in $n\geq n_0$ and $\sfx,\sfy,\sfz,\sfw\in (0, \eta\sqrt{n}]\cap\bbN$.
\end{proposition}
Proposition~\ref{prop:RW-estimate} is a statement about random walks with
transition probabilities $\RWP$. We relegate the proof to the Appendix and
proceed with the proof of Proposition~\ref{prop:gs-lN}.

First of all pick $n =H_\lambda^2$ and note that, in view of
Assumption~\eqref{eq:HL-2}, the following happens:
For any $\sfx, \sfy \leq \eta H_\lambda$ and any path $\bbX\in
\calR_{n}^{\sfx,\sfy}[\eta]$, the value of the potential satisfies
\be{eq:eta-role}
0 \leq
\sum_1^n V_\lambda (X_i )\leq n V_\lambda  (2\eta\sqrt{n}) =
H_\lambda^2 V_\lambda (2\eta H_\lambda )  \leq 2 q(2\eta ) ,
\ee
for all $\lambda$ sufficiently small. In fact, \eqref{eq:eta-role} was precisely the
reason to introduce the  notion of $\eta$-good intervals. An immediate consequence of
\eqref{eq:RW-estimate} and \eqref{eq:eta-role} is that
\be{eq:RW-estimate-l}
\bbP_{H_\lambda^2,+,\lambda}^{\sfx,\sfy}\otimes\bbP_{H_\lambda^2,+,\lambda}^{\sfz,\sfw}
(\exists i \,:\, X^1_i = X^2_i \given
\calR_{H_\lambda^2}^{\sfx,\sfy}\times\calR_{H_\lambda^2}^{\sfz,\sfw}) \geq p {\rm e}^{-4 q (2\eta )},
\ee
for all $\lambda$ sufficiently small. The formula~\eqref{eq:RW-estimate-l}
provides a uniform lower bound on probability of coupling inside an $\eta$-good
interval.

Consider the product measures
$\bbP_{N,+,\lambda_N}^{0,0}\otimes\bbP_{N,+,\lambda}^{\sfu,\sfv}$.
Let $\calM$ be the event that the paths $X^1$ and $X^2$ meet both on the
left and on the right of the segment $[-T,T]$. It follows from
Lemma~\ref{lem_ManyEtaGoodIntervals} and \eqref{eq:RW-estimate-l}
that there exist $K = K(C,T)$ and  $c_5>0$ such that
\be{eq:coup}
\bbP_{N,+,\lambda}^{0,0}\otimes\bbP_{N,+,\lambda}^{\sfu,\sfv} (\calM)
\geq 1-e^{-c_5 N\Hla^{-2}} ,
\ee
uniformly in $\lambda$ small, $\sfu , \sfv \leq C\Hla$ and $N > (K + T)\Hla$.

For $\ell < -T\Hla^2 , r >T\Hla^2$ and $\sfx,\sfy\in\bbN$, let
$\calM_{\ell,r}^{\sfx,\sfy}\subset\calM$ be the event that $\ell$ is
the leftmost meeting point of $X^1$, $X^2 $, and $X^1_\ell = X^2_\ell = \sfx$,
whereas $r$ is the rightmost meeting point of $X^1$, $X^2$, and
$X^1_r = X^2_r = \sfy$. In this notation, $\calM$ is the disjoint union,
$\calM = \cup\calM_{\ell,r}^{\sfx,\sfy}$.

Let $\calA\in\calF_{\lambda,T}$. Then,
\begin{align*}
\bbP_{N,+,\lambda}^{\sfu,\sfv}(\calA) &=
\bbP_{N,+,\lambda}^{0,0}\otimes\bbP_{N,+,\lambda}^{\sfu,\sfv}(\Omega\times\calA)
\\
&=
\bbP_{N,+,\lambda}^{0,0}\otimes\bbP_{N,+,\lambda}^{\sfu,\sfv}
(\Omega\times\calA;\calM^{c}) + \sumtwo{\ell,r}{\sfx,\sfy}
\bbP_{N,+,\lambda}^{0,0}\otimes\bbP_{N,+,\lambda}^{\sfu,\sfv}
(\Omega\times\calA;\calM^{\sfx,\sfy}_{\ell,r}) .
\end{align*}
However,
\[
\bbP_{N,+,\lambda}^{0,0}\otimes\bbP_{N,+,\lambda}^{\sfu,\sfv}
(\Omega\times\calA;\calM^{\sfx,\sfy}_{\ell,r}) =
\bbP_{N,+,\lambda}^{0,0}\otimes\bbP_{N,+,\lambda}^{\sfu,\sfv}
(\calA\times\Omega;\calM^{\sfx,\sfy}_{\ell,r}) .
\]
Therefore ,
\[
\abs{\bbP_{N,+,\lambda}^{\sfu,\sfv} (\calA) - \bbP_{N,+,\lambda}^{0,0} (\calA)}
\leq
\bbP_{N,+,\lambda}^{0,0}\otimes\bbP_{N,+,\lambda}^{\sfu,\sfv}
(\calM^{c}) ,
\]
which, in view of Remark~\ref{rem:P-l-cut} and~\eqref{eq:coup},
implies~\eqref{eq:gs-lN}.
\qed
\end{proof}

\appendix
\section{Proof of Proposition~\ref{prop:RW-estimate}}
We shall employ here our original notation $\RWP$ for the path measure of the
underlying random walk. Our argument is based on the second moment method, which
is put to work using the following input from~\cite{CIL,CCh13}:

\paragraph{Bounds on probabilities of random walks to stay positive}
There exists $\eta_0$, such that for any $\eta \geq \eta_0$ the following
happens: One can find $n_0 = n_0(\eta)$, $c_1 = c_1(\eta)$ and
$c_2 = c_2(\eta)$ such that
\be{eq:CIL}
c_1\frac{\sfx \sfy}{n^{3/2}} \leq \RWP(\calR_{n,m}^{\sfx,\sfy}\left[\eta\right])
\leq c_2\frac{\sfx \sfy}{n^{3/2}} ,
\ee
uniformly $n \geq n_0$, $\frac{n}{3}\leq m\leq n$ and
$1\leq \sfx,\sfy\leq\eta\sqrt{n}$.
Indeed, in view of the invariance principle for random walk
bridges~\cite[Theorem~2.4]{CCh13}, the restriction $X_i\leq 2\eta\sqrt{n}$
may be removed from $\calR_{n,m}^{\sfx,\sfy}[\eta]$ in the following sense:
There exists $c = c(\eta) \in [1,\infty)$, such that
\be{eq:nu-plus}
1\leq \frac{\RWP(\calR_{n,m}^{\sfx,\sfy , +})}{\RWP(\calR_{n,m}^{\sfx,\sfy}
[\eta ])}
\leq c(\eta ) ,
\ee
uniformly in $n \geq n_0$, $\frac{n}{3}\leq m\leq n$ and
$\sfx,\sfy\leq\eta\sqrt{n}$.
Hence, the two-sided inequality~\eqref{eq:CIL} can be verified along the lines
of the
proof of Theorem~4.3 in \cite{CIL}, where a stronger asymptotic statement was
derived for a more restricted range of parameters.

\medskip
Let $1\leq \sfx,\sfy,\sfz,\sfw\leq\eta\sqrt{n}$ and consider now the product measure,
\[
\RWP\otimes\RWP \lb\cdot \given \calR_n^{\sfx,\sfy}\left[\eta\right]\times\calR_n^{\sfz,\sfw}
\left[\eta\right]
\rb
.
\]
Let $\calN$ be the number of intersections of the two replicas $X^1$ and $X^2$
during the time interval $[\frac{n}{3},\frac{2n}{3}]$ below level
$\eta\sqrt{n}$.
Precisely,
\be{eq:N-int}
\calN = \#\bsetof{\ell\in [\frac{n}{3},\frac{2n}{3}]}
{X^1_\ell = X^2_\ell \leq \eta\sqrt{n}}
=
\sum_{\ell=\frac{n}{3}}^{\frac{2n}{3}} \sum_{\sfu\leq\eta\sqrt{n}}
\1_{\{ X^1_\ell = X^2_\ell = \sfu\}} .
\ee
\paragraph{Lower bound on the expectation
$\RWP\otimes\RWP (\calN \given \calR_n^{\sfx,\sfy}\times\calR_n^{\sfz,\sfw})$}
The expectation
\[
\RWP\otimes\RWP (\calN \given \calR_n^{\sfx,\sfy}\times\calR_n^{\sfz,\sfw}) =
\sum_{\ell=\frac{n}{3}}^{\frac{2n}{3}} \sum_{\sfu\leq\eta\sqrt{n}}
\Phi_n(\ell,\sfu;\sfx,\sfy,\sfz,\sfw) ,
\]
where (see Figure~\ref{fig:E1})
\begin{figure}
\begin{center}
\scalebox{.5}{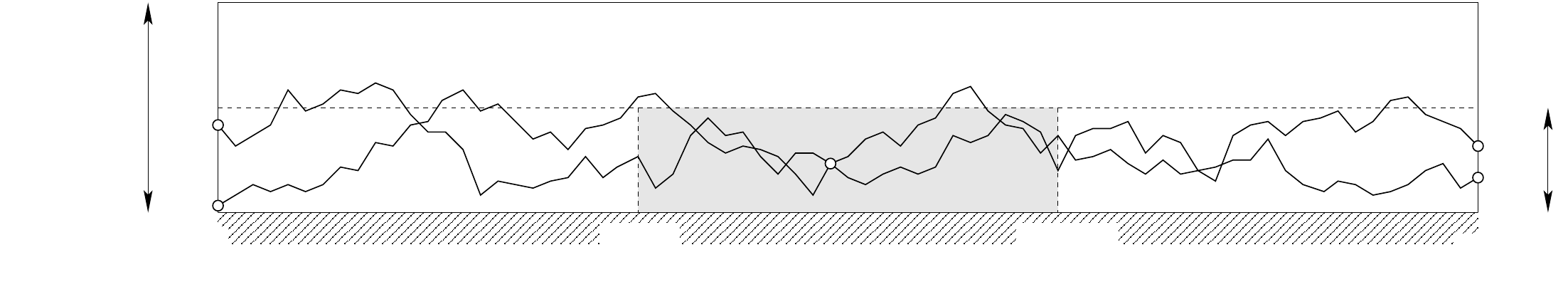}
\end{center}
\caption{The decomposition of
$\RWP\otimes\RWP (\calN \given \calR_n^{\sfx,\sfy}\times\calR_n^{\sfz,\sfw})$.}
\label{fig:E1}
\end{figure}
\[
\Phi_n(\ell,\sfu;\sfx,\sfy,\sfz,\sfw) =
\frac{
\RWP(\calR_{n,\ell}^{\sfx,\sfu})
\RWP(\calR_{n,\ell}^{\sfz,\sfu})
\RWP(\calR_{n,n-\ell}^{\sfu,\sfy})
\RWP(\calR_{n,n-\ell}^{\sfu,\sfw})
}{
\RWP(\calR_{n,\ell}^{\sfx,\sfy})
\RWP(\calR_{n,\ell}^{\sfz,\sfw})
} .
\]
By~\eqref{eq:CIL},
\[
\Phi_n (\ell,\sfu;\sfx,\sfy,\sfz,\sfw) \geq c_3 \frac{\sfu^4}{n^3} ,
\]
uniformly in all the arguments in question. Consequently,
\be{eq:EN-lb}
\RWP\otimes\RWP (\calN \given \calR_n^{\sfx,\sfy}\times\calR_n^{\sfz,\sfw})
\geq c_4 (\eta )\sqrt{n} ,
\ee
also uniformly in $\sfx,\ldots,\sfw \leq \eta\sqrt{n}$.

\paragraph{Upper bound on the expectation
$\RWP\otimes\RWP(\calN^2 \given\calR_n^{\sfx,\sfy}\times\calR_n^{\sfz,\sfw})$}

The expectation
\[
\RWP\otimes\RWP (\calN^2 \given \calR_n^{\sfx,\sfy}\times\calR_n^{\sfz,\sfw})
\leq
\sum_{\substack{\ell,m=\frac{n}{3}\\\ell\leq m}}^{\frac{2n}{3}}
\sum_{\sfu,\sfv\leq \eta\sqrt{n}}
\Psi_n(\ell,\sfu;m,\sfv;\sfx,\sfy,\sfz,\sfw),
\]
where (see Figure~\ref{fig:E2})
\begin{figure}
\begin{center}
\scalebox{.5}{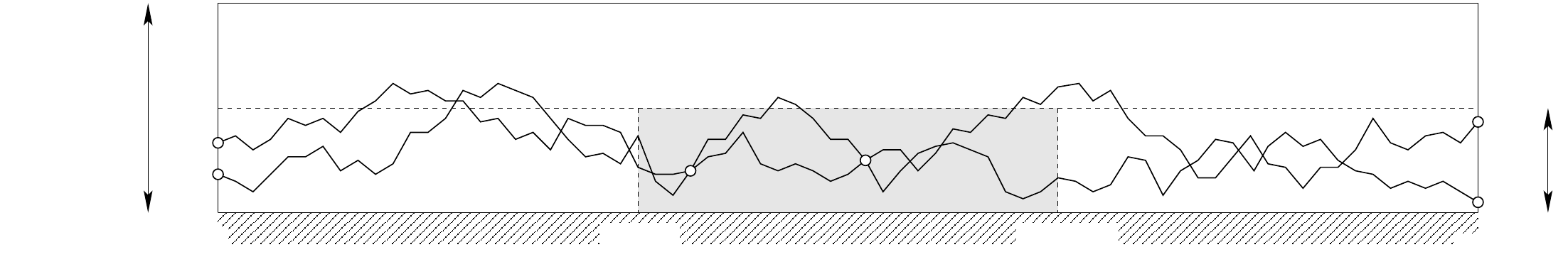}
\end{center}
\caption{The decomposition of
$\RWP\otimes\RWP (\calN^2 \given
\calR_n^{\sfx,\sfy}\times\calR_n^{\sfz,\sfw})$.}
\label{fig:E2}
\end{figure}
\be{eq:Psi}
 \Psi_n (\ell ,\sfu ; m, \sfv ; \sfx, \sfy, \sfz, \sfw ) =
 \frac{\sfp (\calR_{n, \ell}^{\sfx , \sfu })
 \sfp (\calR_{n, \ell}^{\sfz , \sfu }) \sfp_{m-\ell} (\sfu , \sfv )^2
 \sfp (\calR_{n, n-m}^{\sfv, \sfy })
  \sfp (\calR_{n, n-m}^{\sfv , \sfw })
 }
 {\sfp (\calR_{n, \ell}^{\sfx , \sfy })
 \sfp (\calR_{n, \ell}^{\sfz , \sfw })} .
\ee
Above, $\RWP_r(\sfu,\sfv)$ is a short-hand notation for
$\RWP(X_r = \sfv \given X_0=\sfu)$. The inequality in~\eqref{eq:Psi}
is due to the fact that we ignore the positivity condition on the interval
${\ell,\ldots,m}$.

By~\eqref{eq:CIL},
\[
\Psi_n(\ell,\sfu;m,\sfv;\sfx,\sfy,\sfz,\sfw)
\leq c_6(\eta) \frac{\RWP_{m-\ell}(\sfu,\sfv)^2}{n} ,
\]
uniformly in all the arguments in question. Consequently,
\[
\sum_{\substack{\ell,m =\frac{n}{3}\\\ell\leq m}}^{\frac{2n}{3}}
\sum_{\sfu,\sfv\leq\eta\sqrt{n}}
\Psi_n(\ell,\sfu;m,\sfv;\sfx,\sfy,\sfz,\sfw)
\leq c_7 \sqrt{n} \sum_{r=0}^{\frac{n}{3}}
\sum_\sfv \RWP_r (0,\sfv)^2 .
\]
The double sum on the right-hand side above is equal to the expectation
of the number of intersections of two independent $\sfp$-walks during the first
$\frac{n}{3}$ steps of their life. It is bounded above by $c_8\sqrt{n}$.
We conclude that
\be{eq:EN2-up}
\RWP\otimes\RWP (\calN^2 \given \calR_n^{\sfx,\sfy}\times\calR_n^{\sfz,\sfw})
\leq c_9 n .
\ee

\medskip
The lower and upper bounds~\eqref{eq:EN-lb} and~\eqref{eq:EN2-up} imply the
existence of $\nu=\nu(\eta)>0$, such that
\be{eq:CZ}
\RWP\otimes\RWP (\calN^2 \given \calR_n^{\sfx,\sfy}\times\calR_n^{\sfz,\sfw})
\leq
\nu [\RWP\otimes\RWP (\calN \given \calR_n^{\sfx,\sfy}\times\calR_n^{\sfz,\sfw})
]^2 ,
\ee
uniformly in $\sfx,\sfy,\sfz,\sfw\leq\eta\sqrt{n}$.
Set $\calE = \RWP\otimes\RWP (\calN \given
\calR_n^{\sfx,\sfy}\times\calR_n^{\sfz,\sfw})$.
By the Paley-Zygmund inequality,
\be{eq:CZ-input}
\RWP\otimes\RWP (\calN > \alpha\calE \given
\calR_n^{\sfx,\sfy}\times\calR_n^{\sfz,\sfw})
\geq \frac{(1-\alpha)^2}{\nu} ,
\ee
for every $\alpha\leq 1$. \eqref{eq:RW-estimate} follows.

\bibliographystyle{plain}
\bibliography{ISV14-final}

\end{document}